\documentclass[11pt,reqno]{amsart}
\usepackage{amsmath}
\usepackage{amsfonts}
\usepackage{dsfont}
\usepackage{mathrsfs}
\usepackage{amssymb}
\usepackage{amsthm}
\usepackage{enumitem}
\usepackage{bbm}
\usepackage{times}
\usepackage{tabularx}
\usepackage{amsbsy}

\usepackage{mathtools}
\usepackage{tikz}
\usetikzlibrary{arrows,decorations.markings}
\usetikzlibrary{matrix}
\usetikzlibrary{graphs}
\usetikzlibrary{backgrounds}
\usepackage{setspace}     \spacing{1}
\usepackage{color}

\usepackage[colorlinks=true,linkcolor=blue,citecolor=blue]{hyperref}
\usepackage{amssymb,latexsym}
\usepackage{mathrsfs,hyperref}
\usepackage{amsmath,latexsym}
\usepackage{amscd,latexsym} 
\usepackage{float}
\usepackage{indentfirst}
\usepackage{amsfonts, latexsym}

\newtheorem{thm}{Theorem}
\newtheorem{prop}{Proposition}
\newtheorem{lem}[prop]{Lemma}

\newtheorem{clm}[prop]{Claim}

\theoremstyle{definition}

\newtheorem{df}[prop]{Definition} 

\theoremstyle{remark}

\newtheorem{rmk}[prop]{Remark} 

\newcommand{\R}{{\mathbb{R}}}
\newcommand{\Z}{{\mathbb{Z}}}

\newcommand{\Sk}{\mathcal{S}_k}

\newcommand{\M}{\mathcal{M}}

\newcommand{\K}{\mathbb{K}}

\newcommand{\g}{\left\langle\cdot,\cdot\right\rangle}

\newcommand{\al}{\alpha}

\newcommand{\Om}{\Omega}
\newcommand{\ga}{\gamma}
\newcommand{\eps}{\epsilon}

\newcommand{\T}{\mathbb{T}}
\newcommand{\pa}{\partial}

\newcommand{\nat}{\nabla_t}
\newcommand{\bbN}{\mathbb{N}}

\def\H{\rm H}
\def\La{\Lambda_\alpha}
\def\dx{\dot{x}}

\numberwithin{equation}{section}

\title[Infinitely many noncontractible closed magnetic geodesics on non-compact manifolds]{Infinitely many noncontractible closed magnetic geodesics on non-compact manifolds}

\author{Wenmin Gong }

\address{School of Mathematical Sciences, Beijing Normal University,
	Beijing, 100875, China}

\email{ wmgong@bnu.edu.cn}
\begin{document}
\maketitle

\begin{abstract}
	In this paper we study the existence and multiplicity of periodic orbits of exact magnetic flows with energy levels above the Ma\~{n}\'{e} critical value of the universal cover on a non-compact manifold from the viewpoint of Morse theory.
\end{abstract}

\maketitle

\section{Introduction}\label{sec:1}
In this paper we study the existence and multiplicity of periodic orbits of exact magnetic flows with prescribed energy levels on a non-compact and complete Riemannian manifold. On compact Riemannian manifolds these kinds of problems have been studied  in quantity and by different approaches,  for instance the Morse-Novikov theory for (possibly multi-valued) variational functionals~\cite{No, NoT, Ta, Ta1}, the Aubry-Mather theory~\cite{CMP},  the methods from symplectic geometry~\cite{Ar, Ar1, Gi1,Gi2, GiK, HoV, Po,Lu, FS, Me2, Go1,GoX},  the heat flow method~\cite{CJSZ}, the degree theory for immersed closed curves~\cite{RoS, Sc1, Sc2, Sc3} , the variational methods~\cite{Me1, Co1,Go} and so on.

Let $(M,\g)$ be a complete, non-compact and $m$-dimensional Riemannian manifold without boundary. Let $\pi:TM\to M$ be the canonical projection.  A closed $2$-form $\Omega$ gives an endomorphism $Y:TM\to TM$, which is called a \emph{Lorentz force},  as
\begin{equation}\label{df:Loforce}
\Omega_x(u,v)=\langle Y_x(u),v\rangle_x\quad\forall x\in M \;\hbox{and}\;\forall u,v\in T_xM.
\end{equation}
The \emph{magnetic flow} of  the pair $(\g,\Omega)$ is the flow of the second order ODE:
\begin{equation}\label{e:Lorentzforce}
\nat\dot{\ga}=Y(\dot{\ga}),
\end{equation}
where $\ga:\R\to M$ is a smooth curve, and $\nat$ denotes the Levi-Civita covariant derivative. If $\Omega$ is an exact $2$-form, the flow defined by~(\ref{e:Lorentzforce}) is called an \emph{exact magnetic flow}, otherwise it is called a \emph{non-exact magnetic flow} (corresponding to the case of magnetic monopoles). In dimension two these flows are models for the motion of a particle under the Lorentz force as proposed by V. I. Arnold  ~\cite{Ar}.  A solution $\ga:\R\to M$ of (\ref{e:Lorentzforce}) is called a \emph{magnetic geodesic}.

For an exact $2$-form $\Om=-d\theta$ with $\theta$ a smooth $1$-form on $M$, the magnetic flow can be obtained as the Euler-Lagrange flow of the Lagrangian
\begin{equation}\label{e:lagrangian}
L(x,v)=\frac{1}{2}\langle v,v\rangle_x+\theta_x(v).
\end{equation}
The energy associated to the Lagrangian $L$ is defined by
$$E(x,v)=\partial_vL(x,v)[v]-L(x,v)=\frac{1}{2}\langle v,v\rangle_x,$$
and the action of the Lagrangian $L$ over an absolutely continuous curve $\ga:[a,b]\to M$
is given by
$$S_L(\ga)=\int^b_a L(\ga(t),\dot{\ga}(t))dt.$$

A closed magnetic geodesic with energy $k$ can be seen as a critical point of the functional $$S_{L+k}=\int^T_0L(\ga(t),\dot{\ga}(t))dt+kT$$
on the space of periodic curves $\ga:\R\to M$ of arbitrary period $T$. To equip this space with a
differentiable structure such that the infinite dimensional Morse theory works, we denote $W^{1,2}(\T,M)$ the space of $1$-periodic curves on $M$ of Sobolev class $W^{1,2}$, where $\T=\mathbb{R}/\mathbb{Z}$. The functional
$$\Sk:W^{1,2}(\T,M)\times \R^+\to\R,\quad
\Sk(x,T)=T\int_0^1\big(L(x,\dx/T)+k\big)dt$$
is smooth and its critical points $(x,T)$ exactly correspond to the periodic magnetic orbits $\ga(t):=x(t/T)$ of energy $k$, where $\R^+:=(0,\infty)$. In contrast to the geodesic case, we note that magnetic geodesics on different energy levels are not reparametrizations of
each other.

As a comparison, we know that closed geodesics on a non-compact manifold do not always exist, for instance on a Euclidean space, complete Riemannian manifolds of positive curvature~\cite{GrM} and $\R\times M$ with a warped product metric~\cite{Th} given by
$$\langle X,Y\rangle:=\alpha\beta+e^rg(u,v),\quad X=(\alpha,u),\;Y=(\beta,v)\in T_{(r,x)}(\R\times M),$$
where $g$ is a complete Riemannian metric on $M$. However, if some geometric or topological restrictions are put on a non-compact complete Riemannian manifold,  closed geodesics  may occur, for instance, a non-compact surface  with a complete metric neither homeomorphic to the plane nor the cylinder~\cite{Th}, a complete surface of finite area homeomorphic to a cylinder or to a M\"{o}bius band~\cite{Ba} and complete Riemannian manifolds with non-positive sectional curvatures and a non-trivial singular homology group of the loop space in a high degree~\cite{BeG,BGS}. Observe that the magnetic flow of $(\g,\Omega)$ with $\Omega=0$ is the geodesic flow of the Riemannian metric~$\g$. In view of this, in order to find closed magnetic geodesics on a non-compact manifold, it is reasonable to impose some suitable geometric or topological restrictions on the manifold, see, for instance~\cite{BPV,BPRV,MZ}.   In fact, without these restrictions there exist non-compact Riemannian manifolds on which there are no closed magnetic orbits of high energy levels, see Appendix~\ref{app:Nonorbits}. This is strongly in contrast to the case that exact magnetic flows on a closed Riemannian manifold always have closed magnetic geodesics in high energy levels~\cite{Co1, CIPP2}.

Applying Morse theory to obtain critical points of the functional $\Sk$ requires some compactness conditions for the sublevel sets of $\Sk$, for instance the Palais-Smale condition. However, in our setting where $M$ is non-compact, for a fixed energy $k$, in general, $\Sk$ does not satisfy the Palais-Smale condition. The reason comes from two aspects: \textbf{(I)} there may be Palais-Smale sequences of loops which escape towards the ends of the manifold $M$; \textbf{(II)} even on a compact manifold the functional $\Sk$ could not satisfy the Palais-Smale condition for the energy below some threshold and fails to be bounded from below. To overcome the problem caused by \textbf{(I)}, we use the technique of penalized function which was introduced by Benci and Giannoni~\cite{BeG}.

In order to circumvent the problem caused by \textbf{(II)}, we use a dynamical notion of Euler-Lagrange flow associated to a general convex suplinear Lagrangian $L:TM\to\R$ as follows:
$$ c(L):=\inf\big\{k\in\R|S_{L+k}(\ga)\geq0,\quad \forall \ga\in \mathcal{C}(M)\big\},$$
where $\mathcal{C}(M)$ denotes the set of absolutely continuous closed curves in $M$.
The real number defined above is called \emph{Ma\~{n}\'{e} critical value}, which was introduced by Ma\~{n}\'{e} in~\cite{Ma0}. There are many other  Ma\~{n}\'{e} critical values associated to different covers of $M$, which we shall define as following. Let $\widehat{M}$ be a covering of $M$ with covering projection $p$, and let $\widehat{L}$ be the lift of the Lagrangian $L$ to $\widehat{M}$ given by $\widehat{L}=L\circ dp$.
Hence, we have a critical value for  $\widehat{L}$. It is immediate that
$$c(\widehat{L})\leq c(L).$$
More generally, if $M_1$ and $M_2$ are coverings of $M$ such that $M_1$ covers $M_2$, then
$$c(L_1)\leq c(L_2).$$
Suppose that $\widetilde{M}$ is the universal covering of $M$ and $\overline{M}$ the abelian covering. The latter is defined as the covering of $M$ whose fundamental group is the kernel of the \emph{Hurewicz homomorphism} $\pi_1(M)\to\hbox{H}_1(M,\R)$. These covers give rise to the critical values
$$c_u(L):=c(\widetilde{L}),\quad\hbox{and}\quad
c_a(L):=c(\overline{L})$$
where $\widetilde{L}$ and $\overline{L}$ are the lifts of $L$ to $\widetilde{M}$ and $\overline{M}$ respectively. Moreover, it holds that
$$c_a(L)\geq c_u(L)\geq -\inf\limits_{x\in M} L(x,0).$$
The above inequalities are not hard to prove, and in general these values are not equal, see~\cite{PP}.

In this paper, we mainly restrict to the critical value $c_u(L)$ for the Lagrangian $L$ as in~(\ref{e:lagrangian}) rather than those critical values associated to other covers of $M$, and in particular we have  $c_u(L)\geq 0$. For more dynamical and geometric properties  about Ma\~{n}\'{e} critical values in the case that the manifold $M$ is compact we refer the reader to \cite{Ma1,Mat,CIPP1, CIPP2, CFP, Co1, Di, Fa}.  We also mention that if the manifold $M$ is complete and non-compact, the Ma\~{n}\'{e} critical value $c_u(L)$ can be still characterized by weak KAM solutions, and by the action potential provided that the Lagrangian satisfies suitable conditions~\cite{Co, FM}. In this case, if $\pi_1(M)$ is \emph{amenable} then $c_u(L)=c_a(L)$.

\medskip

\subsection{Main results}\label{subsec:Main}
For any $T>0$, every map in $C(\R/ T\Z,M)$ represents a homotopy class of free loops in $M$. Note that topological spaces $C(\R/ T\Z,M)$ and $C(\R/ \Z,M)$ are always homeomorphic.  For $k\in\mathbb{N}$, each $\ga\in C(\R/ T\Z,M)$ as a $T$-periodic map $\ga:\R\mapsto M$ representing a free homotopy class $\alpha$ is also viewed as a $kT$-periodic map from $\R$ to $M$ and hence as an element of $ C(\R/ kT\Z,M)$, which is called the \emph{$k$-th iteration of $\ga$} and denoted by $\ga^k$.  This $\ga^k\in C(\R/ kT\Z,M)$ represents a free homotopy class in $M$, denoted by $\al^k$.

For a free homotopy class $\alpha$, we denote by $\La M$ the component of the free loop space $\Lambda M:=C(\R/ \Z,M)$ of $M$ which represents $\al$, and define
\begin{equation}\label{e:minlen}
l_\alpha:=\inf\limits_{\ga\in \La M} \int_{\T}|\dot{\ga}(t)|dt,
\end{equation}
where $|\cdot|$ is the norm with respect to the metric $\g$.

Denote by $\nabla$ the Levi-Civita connection on $M$. Recall that the Riemann curvature tensor $R$ is defined by
$$R(X,Y)Z=\nabla_X\nabla_YZ-\nabla_Y\nabla_XZ-\nabla_{[X,Y]}Z$$
and that the sectional curvature is given by
$$K_\pi:=\frac{\langle R(X,Y)X,Y\rangle}{\langle X,Y\rangle^2-\langle X,X\rangle \langle Y,Y\rangle}.$$

Throughout this paper,  we shall denote  by $K(x)$ the supremum of sectional curvatures, i.e.,
$$K(x):=\sup\{K_\pi|\pi\subset T_xM\}.$$

The aim of this paper is to prove that under suitable geometric and topological hypotheses, for every energy level $k$ above the Ma\~{n}\'{e} critical value $c_u(L)$ there are infinitely many periodic orbits with energy $k$ on a non-compact Riemannian manifold. To do this, we shall first prove the following theorem.

\begin{thm}\label{thm: mainresult1}
	Let $M$ be a complete, non-compact connected Riemannian manifold of dimension~$m$ which is endowed  with a Riemannian metric $\g$ and an exact $2$-form $\Om=-d\theta$.  Let $\al$ be a nontrivial free homotopy class on $M$.
	Assume that
	
	\begin{itemize}
		\item[{\rm (A)}]  $l_\alpha >0$ and that there exists an integer $q>2m+1$ such that
		$$\H_q(\Lambda_\al M;\K)\neq 0$$
		where ${\rm H}_q(\cdot;\K)$ is the $q$-th singular homology group with  coefficients in $\K$;
		\item[{\rm (B)}]
		\begin{gather}
		\|\theta\|_\infty:=\sup\limits_{x\in M}|\theta(x)|<+\infty, \notag\\
		\lim\limits_{d(x,x_0)\to+\infty}|d\theta_x|=0
		\quad\hbox{and}\quad \lim\limits_{d(x,x_0)\to+\infty}|(\nabla d\theta)_x|=0,\notag	
		\end{gather}
		where $x_0$ is a fixed point in $M$;
		\item[{\rm (C)}]
		\begin{equation}
		\limsup\limits_{d(x,x_0)\to+\infty}K(x)\leq 0.\notag
		\end{equation}
		
	\end{itemize}
	
	\noindent Then for every $k\in(c_u(L),\infty)$ there exists at least one closed magnetic geodesic with energy $k$ in $M$ representing $\al$.
\end{thm}

\begin{rmk}\label{rmk: mainresult}
	{\rm If the exact $2$-form $\Om$ has a compact support in $M$, the condition (B) in Theorem~\ref{thm: mainresult1} obviously holds.
	}
\end{rmk}


Our main theorem in this paper is the following.

\begin{thm}\label{thm: mainresult2}
Under the assumptions of Theorem~\ref{thm: mainresult1}, furthermore if we assume that
	\begin{itemize}
		\item[{\rm (A')}] for all $i\in\mathbb{N}$,  $l_{\al^i}>0$ and that there exists an integer $q>2m+1$ such that
$$\H_q(\Lambda_{\al^i} M;\K)\neq 0\quad \forall i\in\mathbb{N}.$$
	\end{itemize}
Then for every $k\in(c_u(L),\infty)$,  either for some $j\in\mathbb{N}$ there exist infinitely many non-iteration closed magnetic geodesics with energy $k$ representing $\al^j$, or for some sequence $\{j_i\}_{i\in\mathbb{N}}$ of infinitely many positive intergers and all $i\in\mathbb{N}$ there exists at least one non-iteration closed magnetic geodesics with energy $k$ representing $\al^{j_i}$.

\end{thm}

Using Theorem~\ref{thm: mainresult1} one can obtain a sequence of closed magnetic geodesics with energy $k\in(c_u(L),\infty)$.  To prove~ Theorem~\ref{thm: mainresult2}  one needs to exclude that these infinitely many critical points are the iterations of finitely many periodic orbits. Our strategy is to argue indirectly by  assuming that there are only finitely many non-iteration closed magnetic geodesics with energy $k$, and consider critical groups of these critical orbits under the iteration map~(\ref{e:iterate}). On the one hand, one would like to use assumption~(A') and some basic properties of critical groups to produce a non-vanishing critical group of  an arbitrarily large $n$-th iteration of some critical orbit.  The main difficulty to do this is the
already mentioned lack of the Palais-Smale condition.
We again appeal to the penalization method as done in the proof of Theorem~\ref{thm: mainresult1}. But this time the penalized functional may have other critical points except for those critical orbits we have assumed, and this makes difficult to extract useful information of those assumed finitely many periodic orbits by use of critical groups. A useful observation is that the critical orbits obtained in Theorem~\ref{thm: mainresult1}
have large Morse index larger than $2m+1$, see Remark~\ref{rem:thm1}. To make use of this point
we establish a lower bound estimate of the sum of dimensions of critical groups of isolated critical orbits belonging to  \emph{$g$-critical sets} by  Marino and Prodi's perturbation method  in~\cite{MP}, see Section~\ref{subsec:Cr}.  On the other hand, we employ the index theory for iterated periodic orbits (see Section~\ref{sec:index}) to show that for large $n$-th iterations of those finitely many critical orbits their critical groups all vanish. This contradiction shows that there are infinitely many geometrically different closed magnetic geodesics with energy  $k\in(c_u(L),\infty)$.

For the case that the non-compact Riemannian manifold $M$ is simply connected or that the energy levels are less than $c_u(L)$, the existence of closed magnetic geodesics with prescribed energy levels on non-compact Riemannian manifolds is not discussed in the present paper because the approach considered here seems to be  incapable of dealing with the mentioned cases. However, under certain assumptions on the geometry and  topology of a non-compact manifold, for almost every $k\in(0,c_u(L))$ the exact magnetic flow associated to a $\mathcal{L}$-shrinking Lagrangian has a contractible periodic  orbit with energy $k$, see~\cite{Go}.

\section{Preliminaries}
	
\subsection{The functional setting}\label{subsec:funcst}
	Let $(M,\g)$ be an $m$-dimensional Riemannian manifold without boundary. By the Nash embedding theorem, one can embed $M$ isometrically in $\R^N$ (for $N$ large enough) which is equipped with the Euclidean metric.
	
	Consider the Sobolev space of loops
	$$W^{1,2}(\T,\R^N)=\bigg\{x:\T\to \R^N\big|x\;\hbox{is absolutely continuous and} \int_{\T}|x'(t)|_{\R^N}^2dt<\infty\bigg\},$$
	where $|\cdot|_{\R^N}$ is the norm induced by the standard inner product $\langle \cdot, \cdot\rangle_{\R^N}$ on $\R^N$.
	
	From now on, we assume that $M$ is a subset of $\R^N$ for some large integer $N$ whose Riemannian structure is induced by the Euclidean metric $\langle \cdot, \cdot\rangle_{\R^N}$. We put
	$$W^{1,2}(\T,M)=\big\{x\in W^{1,2}(\T,\R^N)|x(\T)\subset M \big\}.$$
	Its tangent space at $x\in W^{1,2}(\T,M)$ is given by
	$$T_xW^{1,2}(\T,M)=\big\{\xi\in W^{1,2}(\T,\R^N)|\xi(t)\in T_{x(t)}M,\;\hbox{for all}\;s\in \T \big\}.$$
	Define a Riemannian metric on the space  $W^{1,2}(\T,M)$ by setting
	\begin{equation}\label{e:Riemetric}
	\langle \xi,\eta\rangle_x:=\langle \xi(0), \eta(0)\rangle_{x(0)}+\int_{\T}\langle \nat\xi(t), \nat\eta(t)\rangle_{x(t)} dt,
	\end{equation}
	where $\nat$ denotes the Levi-Civita covariant derivative along $x$.
	Since $M$ is complete, $W^{1,2}(\T,M)$ is a complete Hilbert manifold, see~\cite{Kl,Pa}. In the following we denote $W^{1,2}_\alpha(\T,M)$ the connected component of $W^{1,2}(\T,M)$ which represents the free homotopy class $\al$.

	Equip the space $\M:=W^{1,2}(\T,M)\times \R^+$  with the product metric
	\begin{equation}\label{e:prodmetr}
	\big\langle(\xi,\al),(\eta,\beta)\big\rangle_{(x,T)}=\al\beta+\langle \xi,\eta\rangle_x.
	\end{equation}
   Then $\M$ is a Hilbert manifold with connected component $\M_\alpha$ representing $\alpha$, and the pair $(x,T)\in\M$ corresponds to the $T$-periodic curve $\ga(t)=x(t/T)$.  Observe that the Hilbert manifold $\M$ and components $\M_\alpha$ are not complete with the riemannian metric~(\ref{e:prodmetr}) because they do not contain the points $(x,0)\in W^{1,2}(\T,M)\times\{0\}$ which are at finite distance from $(x,1)$.

	Given $k\in\R$, consider the \emph{free period action functional}
	$$\Sk:W^{1,2}(\T,M)\times \R^+\longrightarrow \R,$$
	\begin{eqnarray}
	\Sk(x,T)&=&\int^1_0 TL\bigg(x(t),\frac{\dot{x}(t)}{T}\bigg)dt+kT,\notag\\
	&=&\int^T_0L(\ga(s),\dot{\ga}(s))ds+kT=S_{L+k}(\ga)
	\end{eqnarray}
	where $\ga(s)=x(s/T)$. It is well known that for Lagrangians $L:TM\to \R$ of the form~(\ref{e:lagrangian}), $\Sk$ is smooth.
	Moreover, $(x,T)$ is a critical point of $\Sk$ if and only if $\ga(t)=x(t/T)$ is a periodic orbit of energy $k$.
	
	Observe that the functional $\Sk$  is invariant with respect to the $\T$-action, i.e.,
	\begin{equation}\label{e:invariant}
	\psi: \T\times \M\longrightarrow \M,\quad \big(t,(x,T)\big)\longmapsto\big(x(\cdot+t),T\big).
	\end{equation}
	So the critical set of $\Sk$ consists of critical orbits $\T\cdot (x,T)$ (also denoted by $\T\cdot \ga$).

\subsection{The penalized functional}

To regain the Palais-Smale condition, we apply Benci and Giannoni's penalization method~\cite{BeG}. Take a family of proper smooth functions $\{f_\sigma:M\to[0,\infty)\}_{\sigma\in \bbN}$ which satisfies $f_1\geq f_2\geq\ldots$,
\begin{equation}\label{e:unbdd}
\lim_{d(x,x_0)\to+\infty}f_\sigma(x)=+\infty,\quad \forall \sigma\in\bbN
\end{equation}
for some fixed point $x_0\in M$, and such that for any compact set $K\subset M$ there exists $\sigma_0=\sigma_0(K)\in\bbN$ such that supp$(f_{\sigma})\cap K=\emptyset$ for all $\sigma\geq \sigma_0$.

Such a family of functions can be easily constructed by using a partition of unity of $M$, see~\cite{BeG}.
The \emph{penalized functionals} are defined as follows:
$$\Sk^\sigma:\M\longrightarrow\R,\quad \Sk^\sigma(x,T)=\Sk(x,T)+f_\sigma(x(0).$$
A direct calculation (see~Lemma~\ref{lem:firstvar}) shows that the critical points $(x,T)$ of $\Sk^\sigma$ correspond to  those curves $\ga(t)=x(t/T)$ which satisfy
\begin{gather}
\nat\dot{\ga}=Y(\dot{\ga}),\label{e:critpt1}\\
\frac{1}{2}|\dot{\ga}(t)|^2= k,\;\forall\;t\in (0,T)\quad \hbox{and} \label{e:critpt2}\\
\dot{\ga}^-(T)-\dot{\ga}^+(0)=-\nabla f_\sigma (\ga(0)).\label{e:critpt3}
\end{gather}
So every such $\ga$ is a closed magnetic geodesic with energy $k$ if and only if $\ga(0)$ is a critical point of $f_\sigma$, in particular whenever $\ga(0)\notin$ supp$(f_\sigma)$.

\subsection{The first and second variations}
In this section, we calculate the first and second variations of the action functional $\Sk^\sigma$. To compute the first one, let $s\mapsto (x_s, b_s)\in \M$ be a curve such that
$$x=x_0,\quad T=b_0,\quad \xi(t)=\frac{\pa x_s}{\pa s}(t)\bigg|_{s=0}\quad\hbox{and}\quad\al=\frac{db_s}{ds}\bigg|_{s=0}.$$
\begin{lem}\label{lem:firstvar}
	\begin{eqnarray}
	\delta\Sk^\sigma(x,T)[\xi,\al]&=&-\frac{\alpha}{2T^2}\int^1_0|\dot{x}(t)|^2dt+\frac{1}{T}\int^1_0\langle\dot{x}(t), \nat\xi\rangle dt\notag\\
	&&+\int^1_0d\theta(\xi,\dot{x})dt+k\alpha+d f_\sigma (x(0))[\xi(0)].
	\end{eqnarray}
\end{lem}
\begin{proof}
	Since
	\begin{equation}\label{e:actionfun}
	\Sk^\sigma(x,T)=\int^1_0\bigg(\frac{1}{2T}|\dot{x}|^2dt+\theta_x(\dot{x})\bigg)dt+kT+f_\sigma (x(0)),
	\end{equation}
	a direct calculation of $\frac{d\Sk^\sigma(x_s,b_s)}{ds}\bigg|_{s=0}$ leads to the desired formula.
\end{proof}

To do the second variation, let us denote by $\hbox{exp}:TM\to M$ the exponential map with respect to the metric $\g$ on $M$. Take a two-dimensional family of curves $(V(r,s),B(r,s))\in\M$ with $(r,s)\in(-\varepsilon,\varepsilon)\times(-\varepsilon,\varepsilon)$ for some positive number $\varepsilon$ such that
\begin{gather}
B(0,0)=T,\quad \frac{\partial B}{\partial r}(0,0)=\alpha\quad\hbox{and}\quad\frac{\partial B}{\partial s}(0,0)=\beta,\\
V(r,s)(t)=\hbox{exp}(x(t),r\xi(t)+s\eta(t)),
\quad \forall t\in\T,
\end{gather}
where $\xi,\eta\in T_xW^{1,2}(\T,M)$. Clearly, we have
$$\frac{\partial V(r,s)}{\partial r}(0,0)=\xi\quad \hbox{and}\quad\frac{\partial V(r,s)}{\partial s}(0,0)=\eta.$$ Now we can compute the second variational derivatives of $\Sk^\sigma$.

\begin{lem}\label{lem:secondvar}
	Suppose that $(x,T)$ is a critical point of $\Sk^\sigma$. Then
	\begin{eqnarray}
	\delta^2\Sk^\sigma(x,T)[(\xi,\al),(\eta,\beta)]&=&\frac{1}{T}\int^1_0\big(\langle\nat\xi,\nat\eta\rangle+\langle R(\xi,\dx)\eta,\dx\rangle\big)dt\notag\\
	&&-\frac{\alpha}{T^2}\int^1_0\langle\nat\eta,\dx\rangle dt-\frac{\beta}{T^2}\int^1_0\langle\nat\xi,\dx\rangle dt\notag\\
	&&+\int^1_0(\nabla_\xi d\theta)_x(\eta,\dx)dt+\int^1_0d\theta_x(\eta,\nat\xi)dt\notag\\
	&&+H_\sigma(x(0))[\xi(0),\eta(0)]+\frac{2\al\beta k}{T},\notag
	\end{eqnarray}
	where $H_\sigma$ is the Hessian operator of $f_\sigma$ with respect to the Riemannian metric~$\g$.
	
\end{lem}

\begin{proof}[Proof of Lemma~\ref{lem:secondvar}]
	To compute the second variation, we calculate
	$\frac{\partial^2\Sk^\sigma(V(r,s),B(r,s))}{\partial r\partial s}\big|_{r=s=0}$.  By (\ref{e:actionfun}), we write
	$$\Sk^\sigma:=E+\Phi+\Psi,$$
	where

\begin{gather}
E(x,T)=\frac{1}{2T}\int^1_0|\dot{x}|^2dt,\quad \Phi(x,T)=\int^1_0\theta_x(\dot{x})dt,\quad
\Psi(x,T)=kT+f_\sigma (x(0)).\notag
\end{gather}

    Denote
	$V(r,s,t)=V(r,s)(t)$. Since the covariant derivation has vanishing torsion and is compatible with the Riemannian metric, it holds that
	$$\nabla_{r}\frac{\partial V}{\partial t}=\nabla_t\frac{\partial V}{\partial r},\quad  \nabla_{s}\frac{\partial V}{\partial t}=\nabla_t\frac{\partial V}{\partial s}\quad\hbox{and}\quad \nabla_X
	\langle Y,Z\rangle=\langle \nabla_XY,Z\rangle+\langle Y,\nabla_XZ\rangle.$$
	Using the above formulas, we compute the second variation of the energy term.
	\begin{eqnarray}\label{e:2varofE}
	\delta^2 E(x,T)[(\xi,\al),(\eta,\beta)]&=&\frac{\partial^2E(V(r,s),B(r,s))}{\partial r\partial s}\bigg|_{r=s=0}\notag\\
	&=&\frac{\partial }{\partial r}\bigg|_{r=0}\frac{\partial }{\partial s}\bigg|_{s=0}\bigg(\frac{1}{2B(r,s)}\int^1_0\big|\partial_tV(r,s,t)\big|^2dt\bigg)\notag
	\\
	&=&\frac{\partial }{\partial r}\bigg|_{r=0}
    \int^1_0 \frac{1}{2B(r,0)}	
    \frac{\partial }{\partial s} \bigg|_{s=0}\big|\partial_tV(r,s,t)\big|^2dt\notag\\
    &&-\frac{\partial }{\partial r}\bigg|_{r=0}
    \int^1_0\frac{\partial_s B(r,0)}{2B(r,0)^2}\big|\partial_tV(r,0,t)\big|^2dt\notag\\
    &=&-\frac{\alpha}{2T^2}\int^1_0\frac{\partial }{\partial s}\bigg|_{s=0}\big|\partial_tV(0,s,t)\big|^2dt\notag\\
    &&+\frac{1}{2T}\int^1_0\frac{\partial }{\partial r}\bigg|_{r=0}\frac{\partial }{\partial s}\bigg|_{s=0}\big|\partial_tV(r,s,t)\big|^2dt\notag\\
    &&-\frac{\partial_r\partial_sB(0,0)T^2-2b\alpha\beta}{2T^4}\int^1_0|\dx(t)|^2dt\notag\\
    &&-\frac{\beta}{2T^2}\int^1_0\frac{\partial }{\partial r}\bigg|_{r=0}\big|\partial_tV(r,0,t)\big|^2dt\notag\\
    &=&\frac{1}{T}\bigg(\int^1_0\big(\langle\nat\xi,\nat\eta\rangle+\langle R(\xi,\dx)\eta,\dx\rangle-\langle\nabla_\xi\eta,\nat\dx\rangle\big)dt\bigg)
    \notag\\
    &&-\frac{\alpha}{T^2}\int^1_0\langle\nat\eta,\dx\rangle dt-\frac{\beta}{T^2}\int^1_0\langle\nat\xi,\dx\rangle dt \notag\\
    &&-\frac{\partial_r\partial_sB(0,0)T-2\al\beta}{2T^3}\int^1_0|\dx(t)|^2dt.
	\end{eqnarray}
Next, we compute the second variation of $\Phi$.
To do this, let us note that the first variation of $\Phi$ is given by
\begin{equation}\label{e:1varofPhi}
\delta\Phi(x)\xi=\int_0^1d\theta(\xi,\dx)dt.
\end{equation}
	By~(\ref{e:1varofPhi}), we have
	\begin{eqnarray}\label{e:2varofPhi}
	\delta^2 \Phi(x,T)[(\xi,\al),(\eta,\beta)]&=&\frac{\partial^2}{\partial r\partial s}\int^1_0\theta_{V(r,s,t)}(\partial_t V(r,s,t))dt\bigg|_{r=s=0}\notag\\
	&=&\frac{\partial }{\partial r}\bigg|_{r=0}\int^1_0d\theta_{V(r,0,t)}(\partial_sV(r,0,t),\partial_t V(r,0,t))dt\notag\\
	&=&\int^1_0\big((\nabla_\xi d\theta)(\eta, \dx)+d\theta(\nabla_\xi\eta,\dx)+d\theta(\eta,\nat\xi))dt.
	\end{eqnarray}
	For the second variation of $\Psi$, a direct calculation shows that
	\begin{eqnarray}\label{e:2varofPsi}
	\delta^2 \Psi(x,T)[(\xi,\al),(\eta,\beta)]=k\partial_r\partial_sB(0,0)+H_\sigma(x(0))[\xi(0),\eta(0)].
	\end{eqnarray}
	From~(\ref{e:critpt1}) and (\ref{e:critpt2}), we deduce that if $(x,T)$ is a critical point of $\Sk^\sigma$, then
	\begin{gather}
	\frac{1}{T}\int^1_0\langle\nat\dx,\zeta\rangle dt=\int^1_0d\theta(\zeta,\dx) dt,\;\forall \zeta\in T_xW^{1,2}\qquad\hbox{and}\label{e:critpt1'}
	\\
\frac{1}{2}|\dx(t)|^2=kT^2,\;\forall t\in(0,1).\label{e:critpt2'}
	\end{gather}
	By (\ref{e:critpt1'}) and (\ref{e:critpt2'}), combining (\ref{e:2varofE}), (\ref{e:2varofPhi}) and  (\ref{e:2varofPsi}) concludes the desired result.

\end{proof}

\section{The Palais-Smale condition}

	\begin{df}
		Let $\mathcal{M}$ be a Hilbert manifold modelled on a Hilbert space $\mathcal{H}$, and $f$ a $C^1$-functional on $\mathcal{M}$. Given $-\infty\leq b<a<+\infty$, if any sequence $\{x_n\}_{n=1}^\infty\subset \mathcal{M}$ satisfying $a\leq f(x_n)\leq b$ and $\|df(x_n)\|_{\mathcal{H}^*}\to0$  has a convergent subsequence then $f$ is said to satisfy the \emph{Palais-Smale condition} on $\{x\in \mathcal{M}|a\leq f(x)\leq b\}$,
		where $\|\cdot\|_{\mathcal{H}^*}$ denotes the  operator norm on $\mathcal{H}$.
	\end{df}
	
In order to obtain the Palais-Smale condition for $\Sk^\sigma$ on $\M_\al$, we begin with the following known result.
	
\begin{lem}\label{lem:bddbelow}
If $k<c_u(L)$, $\inf_{(x,T)\in\M_\al}\Sk(x,T)=-\infty$; if $k\geq c_u(L)$, then $\Sk$ is bounded from below on $\M_\al$.
\end{lem}

\begin{proof}
	The proof of the first part of the above assertion is the same as that of the case that $M$ is compact, we refer to {\cite[Lemma~4.1]{Co1}} for the detailed proof. In order to prove the second part of it, we take a closed curve $\ga\in\M_\al$. If $k< c_u(L)$, then, by definition, there exists a contractible closed curve $\ga_0$  with $\Sk(\ga_0)<0$. Let $\beta$ be a path joining $\ga(0)$ to $\ga_0(0)$. For all $n\in\bbN$ we find that the juxtapositions
	$\widetilde{\ga}_n=\beta\sharp \ga_0^n\sharp\beta^{-1}\sharp\ga$ are within the free homotopy class $\alpha$ and satisfy
	$$\Sk(\widetilde{\ga}_n)=S_L(\beta)+S_L(\beta^{-1})+n\Sk(\ga_0)+S_L(\ga)\stackrel{n}{\longrightarrow} -\infty.$$
	
\end{proof}

\begin{lem}\label{lem:bddaf0}
	Let $\{(x_n,T_n)\}$ be a sequence in $ \M_\al$ such that $\{\Sk^\sigma(x_n,T_n)\}$ is bounded from above. If $l_\al>0$ and $\|\theta\|_\infty<\infty$, then $T_n$ is bounded away from zero.
\end{lem}
\begin{proof}
	By our assumption, we may assume that $\Sk^\sigma(x_n,T_n)\leq A$ for all $n\in\bbN$ and some positive number $A$. Since $f_\sigma\geq 0$, we have
	\begin{eqnarray}
	A\geq\Sk^\sigma(x_n,T_n)&\geq&\frac{1}{2T_n}\int^1_0|\dot{x}_n(t)|^2dt- \int^1_0\|\theta\|_\infty\cdot|\dot{x}_n(t)|dt+kT_n\notag\\
	&\geq&\bigg(\frac{1}{2T_n}-\frac{1}{4}\bigg)\int^1_0|\dot{x}_n(t)|^2dt-\|\theta\|_\infty^2+kT_n,\notag
	\end{eqnarray}
	consequently, if $0<T_n<2$, for all $n$ it holds that
	\begin{eqnarray}\label{e:energybd}
	\int^1_0|\dot{x}_n(t)|^2dt\leq \frac{4T_n}{2-T_n}\big(A+\|\theta\|_\infty^2\big).
	\end{eqnarray}
	By the H\"{o}lder inequality,
	\begin{equation}\label{e:lenthenergyineq}
	\bigg(\int^1_0|\dot{x}_n(t)|dt\bigg)^2\leq\int^1_0|\dot{x}_n(t)|^2dt.
	\end{equation}
	Combining (\ref{e:energybd}) with (\ref{e:lenthenergyineq}) implies that if $T_n\to0$ then the length of $x_n$ goes to zero. This contradicts the assumption that $l_\alpha>0$.	
\end{proof}

\begin{prop}\label{prop:PSc}
If $l_\al>0$ and $\|\theta\|_\infty<\infty$, then for any  $k> c_u(L)$, $\Sk^\sigma$ satisfies the Palais-Smale condition on $\{(x,T)\in\M_\al|\Sk^\sigma(x,T)\leq A\}$.
\end{prop}


\begin{proof}
	Assume that $\{(x_n,T_n)\}$ is a sequence in $ \M_\al$ such that
	\begin{equation}\label{e:PSC}
		 \Sk^\sigma(x_n,T_n)<A\quad \hbox{and}\quad \|d\Sk^\sigma(x_n,T_n)\|\to0.
	\end{equation}
	 We will show that $\{(x_n,T_n)\}$ has a convergent subsequence in $ \M_\al$.
	First of all, we prove that $T_n$ has a convergent subsequence. By Lemma~\ref{lem:bddaf0}, $T_n\geq C$ for some positive constant $C$, so we only need to prove that $T_n$ is bounded from above. Observe that if $k>c_u(L)$ then
	\begin{equation}\notag
	\Sk^\sigma(x_n,T_n)=\mathcal{S}_{c_u(L)}(x_n,T_n)+(k-c_u(L))T_n+f_\sigma(x_n(0)).
	\end{equation}
	
	Since $f_\sigma\geq 0$ on $M$ and $\mathcal{S}_{c_u(L)}\geq-B>-\infty$ for some positive constant $B$  by Lemma~
	\ref{lem:bddbelow},
	it follows that
	\begin{equation}\label{e:Tupperbdd}
	T_n\leq (A+B)/(k-c_u(L)).
	\end{equation}
	Consequently, to complete the proof, we have to prove that $\{x_n\}$ has a convergence subsequence in $W^{1,2}(\T,M)$.
	Set $\ga_n(t):=x_n(t/T_n)$, then we have
	\begin{eqnarray}
	 A\geq\Sk^\sigma(x_n,T_n)&=&\int^{T_n}_0\bigg(\frac{1}{2}|\dot{\ga}_n(t)|^2+\theta(\dot{\ga}_n(t))\bigg)dt+kT_n+f_\sigma(\ga_n(0))\notag\\
	&\geq&\frac{1}{4}\int^{T_n}_0|\dot{\ga}_n(t)|^2dt-T_n\|\theta\|^2_\infty+kT_n+f_\sigma(\ga_n(0)),\notag
	\end{eqnarray}
	and it follows that
	\begin{gather}
	\frac{1}{T_n}\int^1_0|\dot{x}_n(t)|^2dt=\int^{T_n}_0|\dot{\ga}_n(t)|^2dt\leq 4\big(A+T_n\|\theta\|^2_\infty\big),\label{e:deribdd}\\
	f_\sigma(\ga_n(0))\leq A+T_n\|\theta\|^2_\infty \label{e:dombdd}.
	\end{gather}
	By (\ref{e:deribdd}) and the H\"{o}lder inequality, we have
	\begin{equation}\label{e:lgbdd}
	\int^{T_n}_0|\dot{\ga}_n(t)|dt\leq T_n^{\frac{1}{2}}\bigg(\int^{T_n}_0|\dot{\ga}_n(t)|^2dt\bigg)^{\frac{1}{2}}\leq 2\big(AT_n+T_n^2\|\theta\|^2_\infty\big)^{\frac{1}{2}}.
	\end{equation}

    Using the fact that $\|\theta\|_\infty<+\infty$ and (\ref{e:Tupperbdd}), we deduce that $f_\sigma(\ga_n(0))$ and the length of $\ga_n$  are bounded for all $n\in\bbN$. Then by~(\ref{e:unbdd}), there exists a compact set $K_\sigma\subset M$ (independent of $n$) such that all $\ga_n$ are entirely located in $K_\sigma$. Therefore, $(x_n,T_n)$ is a bounded sequence in $W^{1,2}(\T,M)\times\R^+$, and hence, up to considering a subsequence, one can find $(x,T)\in W^{1,2}(\T,M)\times\R^+$ such that
    \begin{gather}
    T_n\stackrel{n}{\longrightarrow} T\qquad \hbox{for some positive number}\;T\label{lim:time}\\
    x_n\stackrel{n}{\longrightarrow} x \quad\hbox{weakly in}\; W^{1,2}(\T,M)\; \hbox{and strongly in}\;L^\infty(\T,M).\label{lim:wk}
    \end{gather}

	In the following we will prove that $\{x_n\}$ converges strongly to $x$ in $W^{1,2}$. Denote $\pi(z):\R^N\to T_zM$ the orthogonal projection onto $T_zM$, and consider
	$$w_n(t)=\pi(x_n(t))[x_n(t)-x(t)].$$
	Notice that $w_n$ is bounded in $T_{x_n}W^{1,2}$ since all $x_n$ belong to the compact set $K_\sigma$ and $\pi(z)$ is smooth with respect to $z\in M$. By  (\ref{e:PSC}), taking $(\xi,\al)=(w_n,0)$ in~Lemma~\ref{lem:firstvar} arrives at
	\begin{eqnarray}\label{lim:deriative}
    \frac{1}{T_n}\int^1_0\langle\dot{x}_n, \dot{w}_n\rangle dt+\int^1_0d\theta_{x_n}(w_n,\dot{x}_n)dt+d f_\sigma (x_n(0))[w_n(0)]\stackrel{n}{\longrightarrow} 0,
	\end{eqnarray}
	where we have used $\langle\dot{x}_n, \nat w_n\rangle=\langle\dot{x}_n, \dot{w}_n\rangle$. Using $\{x_n\}\subset K_\sigma$ again, we see that $|d\theta_{x_n}|$ is bounded independently of $n$, hence by (\ref{lim:time}),  (\ref{lim:wk}) and (\ref{lim:deriative}),  we get
	\begin{equation}\label{lim:L2product}
	\int^1_0\langle\dot{x}_n, \dot{w}_n\rangle dt\stackrel{n}{\longrightarrow} 0.
	\end{equation}
	Set
	\begin{equation}\label{e:un}
	u_n(t)=\pi^{\perp}(x_n(t))[x_n(t)-x(t)]
	\end{equation}
	where $\pi^{\perp}(z)={\rm id}-\pi(z):\R^N\to T_zM$, then we have
	$$x_n-x=w_n+u_n.$$
	By (\ref{lim:L2product}), we immediately obtain
	 \begin{equation}\label{lim:L2product'}
	 \int^1_0 \big(\langle\dot{w_n}, \dot{w}_n\rangle+\langle\dot{u}_n, \dot{w}_n\rangle+\langle\dot{x}, \dot{w}_n\rangle\big)dt\stackrel{n}{\longrightarrow} 0.
	 \end{equation}
	 Since $\{\dot{w}_n\}$ converges weakly to $0$ because of (\ref{lim:wk}), it follows from (\ref{lim:L2product'}) that
	 \begin{equation}\notag
	 \int^1_0 \big(\langle\dot{w}_n, \dot{w}_n\rangle+\langle\dot{u}_n, \dot{w}_n\rangle\big)dt\stackrel{n}{\longrightarrow} 0.
	 \end{equation}
	 Observe that $\langle\dot{x}_n-\dot{x}, \dot{x}_n-\dot{x} \rangle=\langle\dot{w}_n, \dot{w}_n \rangle+\langle\dot{u}_n, \dot{u}_n \rangle+2\langle\dot{w}_n, \dot{u}_n \rangle$.
	 Therefore, to show that $\{x_n\}$ converges strongly to $x$ in $W^{1,2}$, it suffices to prove that
	 \begin{equation}\label{lim:du_nL2}
	 \dot{u}_n\stackrel{n}{\longrightarrow} 0 \quad\hbox{strongly in}\; L^2.
	 \end{equation}
	To do this, by (\ref{e:un}) we write
	\begin{equation}\label{e:du_n}
	 \dot{u}_n(t)=\pi^\perp\big(x_n(t)\big)\big[\dot{x}_n(t)-\dot{x}(t)\big]+d\pi^\perp\big(x_n(t)\big)\big[\dot{x}_n(t)\big]\big[x_n(t)-x(t)\big].
	\end{equation}
   Using (\ref{lim:wk}) and $\{x_n\}\subset K_\sigma$ we deduce that
   \begin{equation}\label{e:1thofdu_n}
   d\pi^\perp\big(x_n(t)\big)\big[\dot{x}_n(t)\big]\big[x_n(t)-x(t)\big]\stackrel{n}{\longrightarrow} 0 \quad\hbox{strongly in}\; L^2
   \end{equation}
	and
    \begin{equation}\label{e:2thofdu_n}
	\pi^\perp\big(x_n(t)\big)\big[\dot{x}_n(t)-\dot{x}(t)\big]=-\pi^\perp\big(x_n(t)\big) \big[\dot{x}(t)\big]\stackrel{n}{\longrightarrow}0, \quad\hbox{strongly in}\; L^\infty,
	\end{equation}
	where we have used $\pi^\perp\big(x(t)\big) \big[\dot{x}(t)\big]=0$ because $\dot{x}(t)\in T_{x(t)}M$ for all $t\in\T$.
	Combining (\ref{e:du_n})-(\ref{e:2thofdu_n})
	leads to (\ref{lim:du_nL2}). This completes the proof.
	
\end{proof}

\section{ Morse theory}\label{sec:morsethm}
In this section we recall some results on Morse theory which will be used in Section~\ref{sec:mainthm}.

\subsection{A Morse theorem}
Consider a $C^2$-manifold $\mathcal{M}$  modelled on a Hilbert space and $f\in C^2(\mathcal{M},\R)$. Let $x\in\mathcal{M}$ be a \emph{critical point} of $f$, meaning that
$$df(x)[v]=0\qquad \forall v\in T_x\mathcal{M},$$
where $df(x)$ is the differential of $f$ at $x$. Recall that the \emph{Hessian} of the function $f$ at $x$ is the second variational derivative of $f$ at $x$, that is to say,
if $(r,s)\mapsto\ga(r,s)\in \mathcal{M}$ is a $C^2$-map which is defined near $(0,0)\in\R^2$ such that
$$\ga(0,0)=x,\quad \frac{\partial\ga}{\partial r}(0,0)=u\in T_x\mathcal{M}\quad\hbox{and}\quad \frac{\partial\ga}{\partial s}(0,0)=v\in T_x\mathcal{M},$$
then
$${\rm Hess}f(x)[u,v]=\frac{\partial^2f(\ga(r,s))}{\partial r\partial s}\bigg|_{(0,0)}.$$

\begin{df}
	Let $x$ be a critical point of $f$. The (strict) Morse index of $x$ is the dimension of the maximal subspace of $T_xM$ where ${\rm Hess}f(x)$ is negative definite; the large Morse index of $x$ is the dimension of the maximal subspace of $T_xM$ where ${\rm Hess}f(x)$ is negative semidefinite.
\end{df}
Hereafter we denote by $m(x)$ and $m^*(x)$ the Morse index  and large Morse index of a critical point $x$ of $f$ respectively, and denote by $f^c$ and $f^c_b$ the sets $\{x\in \mathcal{M}|f(x)\leq c\}$ and $\{x\in \mathcal{M}|b\leq f(x)\leq c\}$ respectively.  If $m^*(x)=0$, we say that $x$ is a \emph{nondegenerate critical point} of $f$.

Now we cite a Morse theorem from~\cite{Vi, BeG}.

\begin{thm}[Generalized Morse theorem]\label{thm:Morsethm}
	
Let $\mathcal{M}$ be a $C^2$-Hilbert manifold	and  $f\in C^2(\mathcal{M},\R)$. Suppose that
\begin{enumerate}
	\item[{\rm (i)}] for every $(a,b)\in\R^2$ with $ a<b$, the set $\{x\in\M|a\leq f\leq b\}$ is complete;
	\item[{\rm (ii)}] $f$ satisfies the Palais-Smale condition on $f^c$ for all $c\in\mathbb{R}$;
	\item[{\rm (iii)}] $f$ is bounded from below on $\mathcal{M}$;
	\item[{\rm (iv)}] for any critical point $x$ of $f$, the linear map associated to ${\rm Hess}f(x)[\cdot,\cdot]$ owns $0$ as an isolated eigenvalue which has finite multiplicity;
	\item[{\rm (v)}]
	${\rm H}_q(\mathcal{M},\K)\neq 0$ for some integer $q\geq 0$, where ${\rm H}_q(\mathcal{M},\K)$ is the $q$-th group of singular homology with coefficient field $\K$.
	
\end{enumerate}
Then there exists a critical point $x_0$ of $f$ corresponding to the critical value
$$c=\inf\limits_{D\in \Gamma_q}\sup\limits_{x\in D}f(x)$$
which satisfies
$$0\leq m(x_0)\leq q\leq m^*(x_0),$$
where
\begin{equation}\label{e:vircycle}
\Gamma_q:=\big\{D\subset \mathcal{M}: \iota_*({\rm H}_q(D,\K))\neq 0\big\}
\end{equation}
with $\iota:D\to \mathcal{M}$ the inclusion map.
		
\end{thm}

The proof of the above theorem is based on the classical Morse theorem (i.e., corresponding to the case that ${\rm Hess}f(x)$ is non-degenerate) combined with a perturbation method. We will not repeat it here, and refer to ~\cite[Theorem~3.4]{BeG} for the detailed proof.

\subsection{Critical groups and $g$-critical sets}\label{subsec:Cr}
Let $\M$ be a $C^2$-Hilbert-Riemannian manifold and $f\in C^1(\M, \R)$ which satisfies the Palais-Smale condition. Denote by $\mathcal{K}(f)$ the set of critical points of $f$. If $f\in C^2(\M, \R)$, for $g\in\bbN\cup\{0\}$, the set $\mathcal{K}_g(f)$ of all critical points of $f$ whose large Morse indices are not less than $g$ is called \emph{$g$-critical set} of $f$,
$$\mathcal{K}_g(f):=\big\{x\in\mathcal{K}(f)\big|m^*(x)\geq g\big\}.$$
If $N\subseteq \mathcal{K}(f)$ is a closed connected submanifold  of $\M$  and $f|_N=\hbox{constant}$, we say that $N$ is a \emph{critical submanifold} of $f$. Let $N\subseteq\M$ be an isolated critical submanifold of $f$ with $f|_N=c$, and let $U$ be a neighborhood of $N$ such that $U\cap \mathcal{K}(f)=N$.
For every $q\in\bbN\cup\{0\}$, the $q$-th critical group with coefficient field $\K$ of $f$ at $N$ is defined by
$$C_q(f,N;\K):=\hbox{H}_q(U\cap f^c,U\cap (f^c\setminus N);\K),$$
where $\hbox{H}_*(X,Y;\K)$ stands for the relative singular homology with coefficient field $\K$. For more background on critical groups, we refer the reader to \cite[Chapter I]{Cha} and references therein. If $x$ is an isolated critical point of $f$ with finite Morse index and large Morse index, then
\begin{equation}\label{e:Crvanish}
C_q(f,x;\K)=0\quad \hbox{for}\quad q\notin \big[m(x),m^*(x)\big],
\end{equation}
see~\cite[Theorem~I.5.1]{Cha}.

\begin{lem}\label{lem:Crgroup}
	Let $f\in C^2(\M,\R)$ which satisfies the Palais-Smale condition on $f^b_a$ with $a<b$. Suppose that $c\in (a,b)$ is an isolated critical value of $f$ in  $[a,b]$  such that (i) $f^{-1}(c)\cap \mathcal{K}(f)=\cup_{i=1}^rN_i$, where
	every $N_i$ is an isolated critical submanifold of $\M$; (ii) the gradient flow of $f$ near every $N_i$ preserves the fibers of the normal bundle of every $N_i$ (which is identified with a tubular neighbourhood of $N_i$ in $\M$). Then
	$$\hbox{H}_q(f^{b},f^{a};\K)=\bigoplus_{i=1}^rC_q(f, N_i;\K)\quad\hbox{for all}\;q\in\bbN\cup\{0\}.$$
\end{lem}
This lemma can be proven, following the proof of ~\cite[Theorem~2.1]{Wa} or~\cite[Theorem~I.4.2]{Cha} without substantial modifications, by the deformation theorem for a pseudo gradient vector field and the excision theorem of the singular homology theory. For the sake of completeness, we include the proof of this lemma in the Appendix~\ref{app:lem9}
by nearly repeating those proofs in the mentioned references.

\begin{thm}\label{thm:Morsethm2} Let $\mathcal{M}$ be a $C^2$-Hilbert manifold	and  $f\in C^2(\mathcal{M},\R)$.
In addition to  the  assumptions (i) -- (iv) in Theorem~\ref{thm:Morsethm},  assume that for $b\in\R$ and some $q_0\in\bbN$,
\begin{equation}\label{e:indexcondition}
f^b\cap \mathcal{K}_{q_0}(f)=\bigcup\limits_{i=1}^rN_i\subseteq \big\{x\in\M\big|f< b\big\},
\end{equation}
where all $N_i$, $i=1,\ldots,r$ are isolated critical submanifolds and satisfy the condition~(ii) of Lemma~\ref{lem:Crgroup}. Then if $\delta>0$ is small enough, there exists a sufficiently small $\eps>0$ and a function $f_\eps\in C^2(\M,\R)$ such that $\|f-f_\eps\|_{C^2(\M,\R)}<\eps$ and
\begin{equation}\label{e:hominequality}
\hbox{dim}\;{\rm H}_q\big((f_\eps)^{b-\delta-\eps};\K\big)\leq\sum\limits_{i=1}^{r}\hbox{dim}\;C_q(f,N_i;\K)\qquad \hbox{for all $q\geq q_0$}.
\end{equation}

\end{thm}

\begin{proof}
	Set $f|_{N_i}=c_i,\;i=1\ldots,r$.
Since all $N_i$ are contained in $\{x\in\M\big|f< b\}$,
$$ b_1:=\max_i c_i<b.$$
By (i) and (ii), the set $\mathcal{K}(f)\cap f^{b}$ is compact. By (iv) the linear map associated to $\hbox{Hess} f(x)$ is a Fredholm operator of index $0$ for all $x\in\mathcal{K}(f)$. Given $b_2\in(b_1,b)$, following the perturbation methods of Marino and Prodi in~\cite{MP}, for every sufficiently small $\epsilon>0$ there exists $f_\epsilon\in C^2(\M,\R)$ and $\eta(\epsilon),\beta(\epsilon)\in\R$
such that the following holds:
\begin{itemize}
	\item[(a)] $\eta(\epsilon)\searrow0$\;
	 and\; $\beta(\epsilon)\searrow b_2$ as $\eps\to 0$,
	\item[(b)] $f_\epsilon(x)=f(x)$ for all $x\in f^{b_2}$ such that either $dist(x,\mathcal{K}(f)\setminus\cup_{i=1}^rN_i)\geq \eta(\epsilon)$ or $dist(x,\cup_{i=1}^rN_i)\leq\eta(\epsilon)$.
	\item[(c)] the critical points of $f_\epsilon$ lying on $f^{b_2}\setminus\cup_{i=1}^rN_i$ are nondegenerate, and the number of these critical points are finite, denoted by $x_1^\epsilon,\ldots,x_l^\epsilon$ with $l=l(\epsilon)$,
	\item[(d)] $|f(x)-f_\epsilon(x)|+\|df(x)-df_\epsilon(x)\|+\|d^2f(x)-d^2f_\epsilon(x)\|\leq\epsilon$\quad $\forall x\in f^{b_2}$,
	\item[(e)] $f_\epsilon$ satisfies the Palais-Smale condition on $\{x\in\M| f_\epsilon\leq \beta(\epsilon)\}$.
	
\end{itemize}
Clearly, \begin{equation}\label{e:cr}
\mathcal{K}(f_\eps)\cap f^{b_2}=\cup_{i=1}^rN_i\cup\{x^\eps_1,\ldots,x^\eps_l\}.
\end{equation}
Now we proceed in two steps. \\
\noindent {\bf Step 1.} We prove that for sufficiently small $\epsilon>0$ ,
	\begin{equation}\label{e:Lmorseindex}
	m(x_i^\eps)=m^*(x_i^\eps)<q_0\quad\hbox{for all}\;i=1,\ldots,l.
	\end{equation}
In fact, for sufficiently small $\epsilon>0$ each $x_i^\eps$ must belong to a set of critical points into which some critical point $y_i\in f^{b_2}$ of $f$ splits by a small perturbation, and hence $dist(x^\eps_i,y_i)\to 0$ as $\epsilon\to 0$. By (\ref{e:indexcondition}) and term~(b) we have $m^*(y_i)<q_0$.

For simplicity we only prove~(\ref{e:Lmorseindex}) for $i=1$ since the proofs  for $i\in\{2,\ldots,l\}$ are similar.

Let $(U,\varphi)$ be a local chart of $\M$ at $y_1$, $V=\varphi(U)\subset \mathbb{E}$ (the Hilbert space where $\M$ is modelled) such that $x^\eps_1\in U$ for $\eps$ small enough . Set
$\tilde{f}:=f\circ \varphi^{-1}$ and $\tilde{f_\eps}:=f_\eps\circ \varphi^{-1}$. Then $\varphi(x_1^\eps)$ and $\varphi(y_1)$ are critical points of $\tilde{f}$ and $\tilde{f_\eps}$ respectively, and satisfy
\begin{equation}\label{e:indexequ}
m(y_1)=\tilde{m}(\varphi(y_1)),\quad m^*(y_1)=\tilde{m}^*(\varphi(y_1)),\quad m(x_1^\eps)=\tilde{m}(\varphi(x_1^\eps)),
\end{equation}
where $\tilde{m}(\varphi(y_1))$ and $\tilde{m}^*(\varphi(y_1))$ are the strict and large Morse index for $\tilde{f}$ at $\varphi(y_1)$ respectively, while $\tilde{m}(\varphi(x_1^\eps))$ is the Morse index of $\tilde{f_\eps}$ at $\varphi(x_1^\eps)$.

We define $L$ and $L_\eps$ as the linear maps associated to the bilinear forms $\hbox{Hess}\tilde{f}(\varphi(y_1))$ and $\hbox{Hess}\tilde{f_\eps}(\varphi(x^\eps_1))$ respectively. From term~(d) we see that
\begin{equation}\label{e:Lnorm}
\|L-L_\eps\|_{\mathcal{L}(\mathbb{E},\mathbb{E})}\stackrel{\eps}{\longrightarrow} 0.
\end{equation}
Here $\mathcal{L}(\mathbb{E},\mathbb{E})$ is the space of the bounded linear maps from $\mathbb{E}$ to  $\mathbb{E}$.

By  assumption~(iv) in Theorem~\ref{thm:Morsethm}, the operator $L$ possesses a smallest positive eigenvalue, denoted by $\mu_+$. Take $\delta>0$ such that
\begin{equation}\label{e:pos}
\mu_+-\delta>0.
\end{equation}
By~(\ref{e:Lnorm}), for sufficiently small $\eps$ we have
\begin{equation}\label{e:Lnorm1}
\|L-L_\eps\|_{\mathcal{L}(\mathbb{E},\mathbb{E})}<\delta.
\end{equation}

Denote by $\mathbb{E}_+$ the subspace of $\mathbb{E}$ on which $L$ is positive definite. Clearly, the quotient space $\mathbb{E}/\mathbb{E}_+$ has dimension $m^*(y_1)$.
Denote by $\|\cdot\|$ the norm of $\mathbb{E}$ and by $\langle\cdot,\cdot\rangle$ its scalar product. Now for sufficiently small $\eps$ we have
\begin{eqnarray}
\inf\big\{\langle L_\eps v,v\rangle\big|v\in\mathbb{E}_+,\;\|v\|=1\big\}&\geq& \inf\big\{\langle L v,v\rangle\big|v\in\mathbb{E}_+,\;\|v\|=1\big\}\notag\\
&&+\inf\big\{\langle L_\eps v-L v,v\rangle\big|v\in\mathbb{E}_+,\;\|v\|=1\big\}\notag\\
&\geq&\mu_+-\|L-L_\eps\|_{\mathcal{L}(\mathbb{E},\mathbb{E})}>\mu_+-\delta>0\notag
\end{eqnarray}
which imply
$$\tilde{m}^*(\varphi(x_1^\eps))\leq \tilde{m}^*(\varphi(y_1)).$$
So by (\ref{e:indexequ}) we have
$$m^*(x_1^\eps)\leq m^*(y_1)$$
provided that $\eps>0$ is small enough. This completes the proof of (\ref{e:Lmorseindex}).

	\noindent {\bf Step 2.} We prove that for every $b_2\in(b_1,b)$, one can choose $0<\eps<(b_2-b_1)/2$ small enough such that
	$$\hbox{dim}\;{\rm H}_q\big((f_\eps)^{b_2-\eps};\K\big)\leq\sum\limits_{i=1}^{r}\hbox{dim}\;C_q(f,N_i;\K).$$
	
	Note that for $0<\eps<(b_2-b_1)/2$, we have
	$$f^{b_1}\subseteq (f_\eps)^{b_2-\eps}\subseteq f^{b_2}\cap  (f_\eps)^{\beta(\eps)}.$$
	It follows from~(\ref{e:cr}) that one can choose  sufficiently small $\eps>0$ such that
$$\mathcal{K}(f_\eps)\cap (f_\eps)^{b_2-\eps}=\cup_{i=1}^rN_i\cup\{z_1,\ldots,z_k\}\subseteq \{x\in\M|f_\eps<b_2-\eps\},$$
where $\{z_1,\ldots,z_k\}$ is a subset of $\{x^\eps_1,\ldots,x^\eps_l\}$ with $k\leq l$.
Consequently, there are finitely many critical values of $f_\eps$ on $ (f_\eps)^{b_2-\eps}$, saying $\ga_1,\ldots,\ga_h$. Take $-\infty <e_0<e_1<\cdots<e_h=b_2-\eps$ such that $\ga_j\in(e_{j-1},e_j)$, $j=1,\ldots, h$.

Recall that for a triple $Z\subseteq Y\subseteq X$ of topological spaces we have
\begin{equation}\label{e:trigonoineq}
\hbox{dim\;H}_q(X,Z;\K)\leq \hbox{dim\;H}_q(X,Y;\K)+\hbox{dim\;H}_q(Y,Z;\K).
\end{equation}

Observe that item~(e) implies that $f_\eps$ satisfies the Palais-Smale condition on a small open neighborhood of $ (f_\eps)^{b_2-\eps}$.
So we have
\begin{eqnarray}
\hbox{dim\;H}_q\big((f_\eps)^{b_2-\eps};\K\big)&=&\hbox{dim\;H}_q\big((f_\eps)^{e_h},(f_\eps)^{e_0};\K\big)\\
&\leq&
\sum_{j=1}^h \hbox{dim\;H}_q\big((f_\eps)^{e_j},(f_\eps)^{e_{j-1}};\K\big)\notag\\
&\leq&\sum_{j=1}^r \hbox{dim\;}C_q(f_\eps,N_j;\K)+\sum_{j=1}^k \hbox{dim\;}C_q(f_\eps,z_j;\K), \label{e:CrNz}
\end{eqnarray}
where in the second inequality we have used Lemma~\ref{lem:Crgroup} since item~(b) implies that $f_\eps=f$ near a small open neighborhood of $N_j,j=1,\ldots, r$. By Step 1., each
critical point $z_j$ has Morse index less than $q_0$.
Then from~(\ref{e:Crvanish}) and (\ref{e:CrNz}) we deduce that
for any $q\geq q_0$
$$
\hbox{dim\;H}_q\big((f_\eps)^{b_2-\eps};\K\big)
\leq\sum_{j=1}^r \hbox{dim\;}C_q(f_\eps,N_j;\K)=\sum_{j=1}^r \hbox{dim\;}C_q(f,N_j;\K).
$$

Finally, choosing $b_2$ sufficiently close to $b$ and $\eps>0$ sufficiently small,  the desired inequality~(\ref{e:hominequality}) follows
immediately.
	
\end{proof}

\section{The Morse index and its iteration inequalities }\label{sec:index}

\subsection{The Morse index}
Denote by
$$m(x,T)\quad\hbox{and}\quad m^0(x,T)$$
the \emph{Morse index} and \emph{nullity} of a critical point $(x,T)$ of $\Sk$.

Consider the fixed period action functional
$$\Sk^T:W^{1,2}(\R/\Z,M)\times \{T\}\longrightarrow \R,\quad x\longmapsto \Sk(x,T),$$
or equivalently this functional is defined as
$$\Sk^T:W^{1,2}(\R/T\Z,M)\longrightarrow \R,\quad \ga\longmapsto S_{L+k}(\ga).$$

The Morse index and nullity of $\Sk^T$ are denoted by $m_T(x)$ and $m^0_T(x)$ respectively. It is a classical result that these two indices are finite since the corresponding Lagrangian $L+k$ is convex, i.e., $L_{vv}>0$.  Clearly, we have that
\begin{equation}\label{e:indexrelation}
0\leq m(x,T)-m_T(x)\leq 1.
\end{equation}
We refer the reader to~\cite{MeP} for the precise relationship between these two indices.

Let $L:TM\to\R$ be a convex superlinear Lagrangian, which means that for every $(x,v)\in TM$, the second derivative along the fibers $L_{vv}(x,v)$ is uniformly positive definite, namely, in linear coordinates on the fiber $T_xM$, there is a $A>0$ such that
	$$w\cdot L_{vv}(x,v)\cdot w\geq A|w|^2\quad \forall (x,v)\in TM\;\hbox{and}\;w\in T_xM,$$
	and for every $R\geq0$ there is a constant $C(R)\geq 0$ such that
	$$L(x,v)\geq R|v|-C(R).$$

Let $H:T^*M\to\R$ be the dual Hamiltonian of $L$, that is,
 $$ H(x,p):=p\big(\mathcal{L}^{-1}(x,p)\big)-L(\mathcal{L}^{-1}(x,p)\big),$$
where $\mathcal{L}:TM\to T^*M$ is the Legendre transform given by $\mathcal{L}(x,v)=\big(x,\partial_vL(x,v)\big)$.
Denote by $X_H$ the Hamiltonian vector field on the cotangent bundle $T^*M$ given by $\omega(X_H,\cdot)=-dH$, where $\omega$ is the canonical symplectic structure of $T^*M$. For a critical point $(x,T)$ of $\Sk$, we define
$$z(t):=\mathcal{L}\big(\ga(t),\dot{\ga}(t)\big),$$
where $\ga(t)=x(t/T)$ is the corresponding $T$-periodic solution of the Euler-Lagrange equation associated to $L$. Set $z_0:=z(0)$. The differential of the Hamiltonian flow $d\phi_H^T(z_0)$ has the eigenvector $X_H(z_0)$ with eigenvalue $1$. Clearly, $T_{z_0}\big(H^{-1}(0)\big)$
is invariant with respect to the linear map $d\phi_H^T(z_0)$. Take a symplectic basis $e_1, f_1, \ldots, e_m, f_m$ in $T_{z_0}T^*M$ which satisfies $e_1=X_H(z_0)$ and
$$T_{z_0}\big(H^{-1}(0)\big)=\hbox{span}\big\{e_1, e_2, f_2, \ldots, e_m, f_m\big\}.$$
Under the above symplectic basis, the linear map $d\phi_H^T(z_0)$ has the matrix representation
\[
d\phi_H^T(z_0)
=
\left(
\begin{array}{cc|ccc}
1 & * & * &\cdots & *\\
0 & 1 & 0 & \cdots & 0 \\\hline
0 & * \\
\vdots & \vdots & & P \\
0 & *
\end{array} \right)
\]
where $P\in\mathrm{Sp}(2m-2)$ is the Poincar\'e return map (see, e.g.~\cite{Lo}) of $\phi_H^T$ at $z_0$, and the entries marked by $*$ stand for some real numbers. It is clear that the kernel of the second differential $d^2\Sk$ contains the vector $(\dot{x},0)$.  A careful calculation shows that
\begin{equation}\label{e:kernel2thSk}
\frac{\hbox{ker}\;d^2\Sk(x,T)}{\hbox{span}\{(\dot{x},0)\}}\cong \hbox{ker}\;(I-P).
\end{equation}
 For a proof of the above equality we refer to~\cite[Proposition~A.3]{AMP}.

The nullity of the fixed period action functional $\Sk^T$ at $\ga$ and the dimension of the kernel of the linear map $I-d\phi_H^T(z_0)$ has the relation
\begin{equation}\label{e:kernel2thSkT}
m^0_T(x)=\hbox{dim}\;\hbox{ker} \big(I-d\phi_H^T(z_0)\big).
\end{equation}
For a proof of this fact we refer the reader to~\cite{Vi1,LoA}.
From (\ref{e:kernel2thSk}) and (\ref{e:kernel2thSkT}) we deduce that
\begin{equation}\label{e:nullrelation}
m^0_T(x)-1\leq m^0(x,T)\leq m^0_T(x)\leq 2\hbox{dim}(M).
\end{equation}

\subsection{Iteration inequalities}

By our convention in Section~\ref{sec:1}, the iteration $\ga^n\in W^{1,2}(\R/nT\Z,M)$ of an element $\ga\in W^{1,2}(\R/T\Z,M)$  gives a smooth action of the multiplicative group $\bbN$ on $\M$ by
\begin{equation}\label{e:iterate}
\Theta:\bbN\times \M\longrightarrow \M,\quad \big(n,(x,T)\big)=(n,\ga)\longmapsto \Theta^n(\ga):=\ga^n=(x^n,nT),
\end{equation}
where $x^n(t)=x(nt)$.

The \emph{mean index} of $\Sk$ is defined as
$$\hat{m}(x,T):=\lim\limits_{n\to \infty}\frac{m(x^n,nT)}{n}.$$
From (\ref{e:indexrelation}) one can see that $\hat{m}(x,T)$ is well defined and coincides with the classical mean index with respect to the fixed action functional, i.e.,
\begin{equation}\label{e:index=}
\hat{m}(x,T)=\hat{m}_T(x):=\lim\limits_{n\to \infty}\frac{m_{nT}(x^n)}{n}.
\end{equation}

For a critical point $(x,T)$ of  $\Sk^T$, we have the iteration inequalities of the Morse index:
\begin{equation}\label{e:indexestimate1}
n\hat{m}_T(x)-\hbox{dim}(M)\leq m_{nT}(x^n)\leq n\hat{m}_T(x)+\hbox{dim}(M)-m_T^0(x^n)
\end{equation}
provided that the pull back bundle $\gamma^*TM\to \R/T\Z$ with $\gamma(t)=x(t/T)$ is trivial.
Since $(\gamma^2)^*TM\to \R/T\Z$  is always trivial, we have that
\begin{equation}\label{e:indexestimate2}
n\hat{m}_{2T}(x^2)-\hbox{dim}(M)\leq m_{2nT}(x^{2n})\leq n\hat{m}_{2T}(x^2)+\hbox{dim}(M)-m_{2T}^0(x^{2n}).
\end{equation}
For a detailed proof of (\ref{e:indexestimate1}) we refer to~\cite{Bo}, see also~\cite{LL,Lo,Lu1}.

Combining (\ref{e:indexrelation}), (\ref{e:nullrelation}), (\ref{e:index=}) and (\ref{e:indexestimate1}) implies the following.
\begin{lem}
	Let $\ga=(x,T)$ be a critical point of $\Sk$ such that the pull back bundle $\gamma^*TM\to \R/T\Z$ is trivial. Then
	\begin{equation}\label{e:iterationineq1}
	n\hat{m}(x,T)-\hbox{dim}(M)\leq m(x^n,nT)\leq n\hat{m}(x,T)+\hbox{dim}(M)-m^0(x^n,nT)+1
	\end{equation}
	Consequently, for any critical point of $\Sk$ it holds that
	\begin{equation}\label{e:iterationineq2}
	n\hat{m}(x^2,2T)-\hbox{dim}(M)\leq m(x^{2n},2nT)\leq n\hat{m}(x^2,2T)+\hbox{dim}(M)-m^0(x^{2n},2nT)+1.
	\end{equation}
	
\end{lem}
Hereafter by a slight abuse of notation if $(x,T)$ is a critical point of $\Sk$ with $\ga(t)=x(t/T)$ we think of $m(\ga)$, $m^0(\ga)$ and $\hat{m}(\ga)$ as $m(x,T)$, $m^0(x,T)$ and $\hat{m}(x,T)$ respectively.

\section{Proof of the main theorem}\label{sec:mainthm}

\subsection{Proof of Theorem~\ref{thm: mainresult1}}

In this subsection, we will apply the generalized Morse theorem in Section~\ref{sec:morsethm} to $\Sk^\sigma$ on $\M_\al$ to prove our main theorem. To this end, we need to check assumptions (i)--(v) of Theorem~\ref{thm:Morsethm}.

First of all, by Lemma~\ref{lem:bddaf0} and Proposition~\ref{prop:PSc},  $\Sk^\sigma$ satisfies (i) and (ii) for every $\sigma\in\bbN$. The boundedness of $\Sk^\sigma$ from below  can be easily seen from
\begin{equation}\label{e:bddbelow}
\Sk^\sigma(x,T)=\mathcal{S}_{c_u(L)}(x,T)+(k-c_u(L))T+f_\sigma(x(0)).
\end{equation}
In fact, for $k>c_u(L)$, $\Sk^\sigma$ is bounded on $\M_\al$ by Lemma~
\ref{lem:bddbelow}. So (iii) is satisfied for $\Sk^\sigma$. For (v), since $\Lambda_\al M$ is homotopically equivalent to $\M_\al$ it holds that ${\rm H}_q(\M_\al,\K)={\rm H}_q(\Lambda_\al M,\K)\neq0$. It remains to verify (iv).

\begin{lem}\label{lem:Hess}
	For all $\sigma\in\bbN$, $\Sk^\sigma$ satisfies (iv) in Theorem~\ref{thm:Morsethm}.
\end{lem}
\begin{proof}
	Given a critical point $(x,T)$ of  $\Sk^\sigma$, define $$A_\sigma:T_{(x,T)}\M_\al\longrightarrow T_{(x,T)}\M_\al$$
	the linear self-adjoint operator by
	$$\big\langle (\xi,\al),A_\sigma (\eta,\beta)\big\rangle_{(x,T)}={\rm Hess}\Sk^\sigma(x,T)\big[ (\xi,\al), (\eta,\beta)\big].$$
	By virtue of Lemma~\ref{lem:secondvar} and (\ref{e:prodmetr}), we see that
	$$A_\sigma=\frac{1}{T}\bigg(id+B_\sigma\bigg)$$
	for which $B_\sigma:T_{(x,T)}\M_\al\longrightarrow T_{(x,T)}\M_\al$ is given by
	\begin{eqnarray}\label{Op:cpt}
	\big\langle (\xi,\al),B_\sigma (\eta,\beta)\big\rangle_{(x,T)}
	&=&\int^1_0 \langle R(\xi,\dx)\eta,\dx\rangle-\frac{\alpha}{T}\int^1_0\langle\nat\eta,\dx\rangle dt \notag\\
	&&-\frac{\beta}{T}\int^1_0\langle\nat\xi,\dx\rangle dt+T\int^1_0(\nabla_\xi d\theta)_x(\eta,\dx)dt\notag\\
	&&+T\int^1_0d\theta_x(\eta,\nat\xi)dt+(2k-1)\alpha\beta
	\notag\\
	&&+T\big(H_\sigma(x(0))[\xi(0),\eta(0)]-\langle\xi(0),\eta(0)\rangle\big).
	\end{eqnarray}
	Therefore, to prove that $\Sk^\sigma$ satisfies assumption (iii) of Theorem~\ref{thm:Morsethm}, it suffices to prove that $B_\sigma$ is a compact operator. To do this, we have to show that for any sequence $\{(\eta_n,\beta_n)\}\subset T_{(x,T)}\M_\al$,
	if $(\eta_n,\beta_n)$ converges to $(\eta_*,T_*)$ weakly  in $W^{1,2}$ and strongly in $L^\infty$ then
	$$\bigg|\big\langle (\xi,\al),B_\sigma (\eta_n,\beta_n)\big\rangle_{(x,T)}-\big\langle (\xi,\al),B_\sigma (\eta_*,\beta_*)\big\rangle_{(x,T)}\bigg| \xrightarrow[\rm uniformly]{n} 0$$
   with respect to $(\xi,\al)\subset T_{(x,T)}\M_\al$ satisfying $\|(\xi,\al)\|\leq 1$, which is easily seen from $(\ref{Op:cpt})$.

\end{proof}

Since assumptions (i)--(v) of Theorem~\ref{thm:Morsethm} are satisfied for $\Sk^\sigma$, we have the following.

\begin{thm}\label{thm:critpt}
	Suppose that $l_\al>0$,  $\|\theta\|_\infty<\infty$, $k>c_u(L)$ and ${\rm H}_q(\Lambda_\al M;\K)\neq0$. Then for every $\sigma\in\bbN$, there exists a critical point $(x_\sigma,T_\sigma)\in\M_\al$ of $\Sk^\sigma$ satisfying
	\begin{gather}
	\Sk^\sigma(x_\sigma,T_\sigma)=\inf\limits_{D\in \Gamma_q}\sup\limits_{(x,T)\in D}\Sk^\sigma(x,T)\qquad\hbox{and}\label{eq:critvalSk1}\\
	q\leq m^0(x_\sigma,T_\sigma)+m(x_\sigma,T_\sigma),\label{eq:critptSk2}
	\end{gather}
where $q\in\bbN$ is given by the assumption~(A), and $\Gamma_q$ is defined as in (\ref{e:vircycle}).
		
\end{thm}

From~(\ref{e:critpt1})--(\ref{e:critpt3}), we see that for every $\sigma\in\bbN$,  $\ga_\sigma(t)=x_\sigma(t/T_\sigma)$ obtained in Theorem~\ref{thm:critpt} is a $C^\infty$ magnetic geodesic on $[0,T_\sigma)$, but not a closed magnetic geodesic because in general $\dot{\ga}_\sigma(0^+)\neq \dot{\ga}_\sigma(T_\sigma^-)$.

\begin{proof}[\textbf{Proof of Theorem~\ref{thm: mainresult1}}] Let $(x_\sigma,T_\sigma)$ be a critical point of $\Sk^\sigma$ given by Theorem~\ref{thm:critpt} which satisfies (\ref{eq:critvalSk1}) and (\ref{eq:critptSk2}).
We shall prove that there exists a compact set $K$ of $M$ independent of $\sigma$ such that $x_\sigma\subset K$ for all $\sigma\in\bbN$. Once this is proved, then for $\sigma$ sufficiently large, we have supp$f_{\sigma}\cap K=\emptyset$. Then it follows from~(\ref{e:critpt3}) that $(x_{\sigma},T_{\sigma})$ corresponds to a closed magnetic geodesic.

Arguing by contradiction, assume that
there exists a subsequence $\sigma_n\stackrel{n}{\to}\infty$ such that
\begin{equation}\label{lim:maxgoinfinity}
\sup\limits_{t\in[0,1]}d(x_{\sigma_n}(t),x_0)\stackrel{n}{\longrightarrow} +\infty
\end{equation}
for some fixed point $x_0$ in $M$. For simplicity, in the following we abbreviate $(x_{\sigma_n},T_{\sigma_n})$ and $\Sk^{\sigma_n}$ by $(x_n,T_n)$ and $\Sk^n$ respectively.

Since any homology class can be represented by a cycle with compact support, there exists a constant $A=A(k)>0$ (independent of $n$) such that
$$\Sk^n(x_n,T_n)\leq A$$
for all $n\in\bbN$. Then by $l_\al>0$,  $\|\theta\|_\infty<\infty$ and $k>c_u(L)$ we have that
\begin{gather}
\frac{A+\max\{0,\inf_{\M_\al}\mathcal{S}_{c_u(L)}\}}{k-c_u(L)}=:d(k)\geq T_n\geq \delta(k):=\frac{2l_\alpha^2}{l_\alpha^2+4A+4\|\theta\|^2_\infty}\qquad\hbox{and}\label{tbddbelow}\\
\int^1_0|\dot{x}_n(t)|dt\leq D(k):=2\big(Ad(k)+d^2(k)\|\theta\|^2_\infty\big)^{\frac{1}{2}}.\label{lenbddabove}
\end{gather}
In fact, (\ref{tbddbelow}) can be seen from (\ref{e:Tupperbdd}), (\ref{e:energybd}) and (\ref{e:lenthenergyineq}), while (\ref{lenbddabove}) from   (\ref{e:lgbdd}).

Since by our notation every $(x_n,T_n)$ is critical point of $\Sk^n$, its speed is independent of time (see~(\ref{e:critpt2'})). So by
(\ref{lenbddabove}) we have that
\begin{equation}\label{eq:velocityestim}
|\dot{x}_n(t)|\leq D\;\hbox{for all}\; t\in [0,1]
\end{equation}

From (\ref{lim:maxgoinfinity}) and (\ref{lenbddabove}) we deduce that
\begin{equation}\label{lim:goinfinity}
\inf\limits_{t\in[0,1]}d(x_{n}(t),x_0)\stackrel{n}{\longrightarrow} +\infty.
\end{equation}
In what follows, we will show that if some geometric restrictions are put on the end of $M$, (\ref{lim:goinfinity}) cannot happen.

For each $n\in\bbN$ we set
$$V^n:=\big\{(\xi,0)\in T_{(x_n,T_n)}\M_\al\big|\xi(0)=0\big\}.$$
It follows from~(\ref{e:Riemetric}) and (\ref{e:prodmetr})   that
$$(V^n)^\perp=\big\{(\eta,\beta)\in T_{x_n}\M_\al\big|\nat\nat\eta(t)=0\;\hbox{for all }t\in [0,1]\;\hbox{with}\;\eta(0)=\eta(1)\;\hbox{and}\;\beta\in\R\big\},$$
consequently,  we have
\begin{equation}\label{eq:dimestim}
\hbox{dim}(V^n)^\perp\leq 2\hbox{dim} T_{x_n(0)}M+1=2m+1.
\end{equation}
Then by Lemma~\ref{lem:secondvar} and (\ref{lim:goinfinity}),
using assumptions~(B) and (C) in Theorem~\ref{thm: mainresult1}, one can find a sequence of numbers $\tau_n\stackrel{n}{\to} 0^+$ such that
for every $(\xi,0)\in V^n$ we have that
	\begin{eqnarray}\label{eq:SkHessestim}
\hbox{Hess}\Sk^n(x_n,T_n)[(\xi,0),(\xi,0)]&=&\frac{1}{T_n}\int^1_0\big(\langle\nat\xi,\nat\xi\rangle+\langle R(\xi,\dx_n)\xi,\dx_n\rangle\big)dt\notag\\
&&+\int^1_0(\nabla_\xi d\theta)_{x_n}(\xi,\dx_n)dt+\int^1_0d\theta_{x_n}(\xi,\nat\xi)dt,\notag\\
&\geq&\frac{1}{T_n}\int^1_0\bigg(\langle\nat\xi,\nat\xi\rangle-K(\xi,\dx_n)\big[|\xi|^2|\dx_n|^2-\langle \xi, \dx_n\rangle^2\big]\bigg)dt\notag\\
&&-\int^1_0|(\nabla d\theta)_{x_n}||\xi|^2|\dx_n|dt-\int^1_0|(d\theta)_{x_n}||\nat\xi||\xi|dt\notag\\
&\geq&\frac{1}{T_n}\int^1_0\langle\nat\xi,\nat\xi\rangle dt-\frac{\tau_n}{T_n}\int^1_0\big[|\xi|^2|\dx_n|^2-\langle \xi, \dx_n\rangle^2\big]dt\notag\\
&&-\tau_n\int^1_0|\xi|^2|\dx_n|dt-\tau_n\int^1_0|\nat\xi||\xi|dt.
\end{eqnarray}
Since $\xi(0)=0$ for all $(\xi,0)\in V^n$, by the H\"{o}lder inequality, it is easy to show that
\begin{equation}\label{eq:poincareineq}
\int^1_0|\nat \xi|^2dt\geq \frac{1}{4}\int^1_0|\xi|^2 dt.
\end{equation}
Plugging (\ref{eq:velocityestim}) and (\ref{eq:poincareineq}) into (\ref{eq:SkHessestim}), for every $(\xi,0)\in V^n$ we have that
\begin{eqnarray}\label{eq:SkHessestim1}
\hbox{Hess}\Sk^n(x_n,T_n)[(\xi,0),(\xi,0)]
&\geq&\frac{1}{T_n}\int^1_0\big|\nat\xi\big|^2dt -\frac{D^2\tau_n}{T_n}\int^1_0|\xi|^2dt-D\tau_n\int^1_0|\xi|^2dt\notag\\
&&-\tau_n\int^1_0\bigg(\frac{1}{T_n}|\nat \xi|^2+\frac{T_n}{4}|\xi|^2\bigg)dt\notag\\
&\geq&\frac{1}{T_n}\big(1-4D^2\tau_n-4D\tau_nT_n-\tau_n-T^2_n\tau_n\big)\int^1_0\big|\nat\xi\big|^2dt.\notag
\end{eqnarray}
Note that $T_n$ is bounded above by some positive constant depending on energy level $k\in(c_u(L),\infty)$ (see~\ref{e:Tupperbdd}). It follows that for large enough $n$
$$\hbox{Hess}\Sk^n(x_n,T_n)[(\xi,0),(\xi,0)]>0\quad \hbox{for all}\; (\xi,0)\in V^n\setminus\{0\}.$$
So we get that
$$m^0(x_n,T_n)+m(x_n,T_n)\leq\hbox{dim}(V^n)^\perp=2m+1$$
which contradicts (\ref{eq:critptSk2}). This completes the proof.
	
\end{proof}
From the proof of Theorem~\ref{thm: mainresult1} we see the following.
\begin{rmk}\label{rem:thm1}
	Given $A>0, k>c_u(L)$ and a free homotopy class $\alpha$ with $l_\al>0$, under the assumptions~(B) and (C) in Theorem~\ref{thm: mainresult1}, there exists a positive integer $\sigma_*=\sigma_*(A,k,l_\alpha)$ such that if $\sigma>\sigma_*$ any critical point $(x,T)$ of $\Sk^\sigma$ representing $\al$ which satisfies $\Sk^\sigma(x,T)<A$ and $m^0(x,T)+m(x,T)>2m+1$ (if there exists) is precisely some critical point of $\Sk$ on $\M_\al$ with supp$f_{\sigma}\cap x(\T)=\emptyset$.
\end{rmk}

\subsection{Proof of Theorem~\ref{thm: mainresult2}}
Our proof of Theorem~\ref{thm: mainresult2} follows some similar ideas in~\cite{Lo1,Lu1}, but also incorporates Benci and Giannoni's penalization method~\cite{BeG}. The key of the proof is to use Theorem~\ref{thm:Morsethm2}.

Under the assumptions of Theorem~\ref{thm: mainresult2}, given $k\in(c_u(L),\infty)$, by Theorem~\ref{thm: mainresult1},  for every $i\in\bbN$ there exists a closed magnetic geodesic $\ga_i$ in $M$ representing $\al^i$ with energy $k$ and large Morse index $m^*(\ga_i)\geq q$ ($> 2m+1$). In other words, for all $i\in\bbN$ the $q$-critical sets $\mathcal{K}_q(\Sk,\al^i) $ of $\Sk$ on $\M_{\al^i}$ are not empty.  However, this sequence of periodic orbits $\{\ga_i\}_{i\in\bbN}$ may only consist of iterations of finitely many non-iteration periodic orbits.


To finish the proof, we argue by contradiction and assume that
\begin{description}
	\item[(H)] (i) for every $i\in\bbN$, $\Sk$ possesses only finitely many non-iteration critical orbits of energy $k$ representing $\al^i$; (ii) there exists an integer $i_0>1$ so that for all $i>i_0$ every closed magnetic geodesic of energy $k$ representing $\al^i$ is an iteration of one of finitely many non-iteration closed magnetic geodesics  $\hat{\ga}_i$ of energy $k$ representing $\al^{j_i}$ with $1\leq j_i\leq i_0$, $i=1,2,\ldots, r$.
\end{description}
Then the critical set of $\Sk$ consists of finitely many critical circles
$$\T\cdot \hat{\ga}_1,\T\cdot \hat{\ga}_1,\ldots,\T\cdot \hat{\ga}_r$$
together with their iterates $\T\cdot \hat{\ga}_i^n$ for $1\leq i\leq r$ and $n\in\bbN$.

Set $\rho:=i_0!$. Define the iteration periodic orbits as
$\ga_i:=\hat{\ga}_i^{\rho/j_i}$, $i=1, 2, \ldots, r$. Clearly, each $\ga_i$ is a closed magnetic geodesic of energy $k$ representing $\al^{\rho}$, and for any $\ell\in\bbN$ we have the following:
\begin{equation}\label{e:iteration}
\mathcal{K}(\Sk, \al^{\ell\rho})=\big\{\T\cdot\ga_i^\ell\big|1\leq i\leq r\big\}
\end{equation}
because of $\ell\rho>i_0$.

 In the following \textbf{we always assume that  the pullback bundles $\ga_i^*TM\to\R/T_i\Z$ are trivial}, where $T_i$, $i=1,2,\ldots,r$ are the  periods of $\ga_i$, otherwise one may consider the $2$-fold iterations $\ga_i^2$ instead of $\ga_i$ and the following argument goes through analogously.

\begin{clm}~\label{clm:novanishing} Under the assumptions of Theorem~\ref{thm: mainresult2},
	for each $\ell\in\bbN$ there exists $\T\cdot\ga_*^\ell \in\mathcal{K}(\Sk, \al^{\ell \rho})$ satisfying
	$$C_q\big(\Sk,\T\cdot\ga_*^\ell;\K\big)\neq0\quad\hbox{and}\quad q-2m\leq q-m^0(\ga_*^\ell)\leq m(\ga_*^\ell)\leq q.$$
\end{clm}

\begin{proof}[Proof of Claim~\ref{clm:novanishing}]
Consider the penalized functional $\Sk^\sigma$ on the Hilbert manifold $\M_{\al^{\ell\rho}}$. By the condition (A') ,
$l_{ \al^{\ell \rho}}>0$ and $\H_q(\Lambda_{\al^{\ell \rho}} M;\K)\neq 0$. Then by Lemma~\ref{lem:bddbelow} and Proposition~\ref{prop:PSc}
for every $k>c_u(L)$ and all $\sigma\in\bbN$ the functionals $\Sk^\sigma:\M_{\al^{\ell\rho}}\to \R$ satisfy the Palais-Smale condition and are bounded from below.

Let $c_1<\cdots<c_l$ be all critical values of $\Sk$ on $\M_{\al^\rho}$, where $l\leq r$. By (\ref{e:iteration}), $\ell c_1<\cdots<\ell c_l$ are all the critical values of $\Sk$ on $\M_{\al^{\ell\rho}}$.

 Let $\sigma_*=\sigma_*(A,k,l_{\al^{\ell\rho}})$ be an integer such that for any $\sigma>\sigma_*$ every critical point $(x,T)$ of $\Sk^\sigma$ representing $\al^{\ell\rho}$ with $\Sk^\sigma(x,T)\leq A$ and $m^*(x,T)>q$ is precisely some critical point of $\Sk$ on $\M_{\al^{\ell\rho}}$, and there is at least one critical point with this property, see Remark~\ref{rem:thm1}. Note also that near a small neighborhood of such a critical point $(x,T)$ we have $\Sk^\sigma=\Sk$, which means that $\T\cdot(x,T)$ is an isolated critical submanifold of $\Sk^\sigma$ because it is of $\Sk$.

 Given $A>\ell c_l$, fix $\sigma>\sigma_*$. Then
 $$\mathcal{K}_{q}(\Sk^\sigma,\al^{\ell\rho})\cap \big(\Sk^\sigma\big)^A=\big\{\T\cdot\ga_{i_j}^\ell\big|1\leq j\leq s\big\}\subseteq \mathcal{K}(\Sk, \al^{\ell\rho}),$$
 where $s\leq r$ and $1\leq i_j\leq r$, $j=1,\ldots, s$.
From the proof of Theorem~\ref{thm: mainresult1} we find that assumptions (i) -- (iv) in Theorem~\ref{thm:Morsethm} for $\Sk^\sigma$ on $\M_{\al^{\ell\rho}}$ are all satisfied. Consequently, by Theorem~\ref{thm:Morsethm2}, for sufficiently small $\delta>0$ and $\eps>0$, there exists a function $f_\eps$ on $\M_{\al^{\ell\rho}}$ such that $\|f_\eps-\Sk^\sigma\|_{C^2(\M_{\al^{\ell\rho}},\R)}<\eps$ and
\begin{equation}\label{e:nontrivialCrgroup}
\hbox{dim}\;{\rm H}_q\big((f_\eps)^{A-\delta-\eps};\K\big)\leq\sum\limits_{j=1}^{s}\hbox{dim}\;C_q(\Sk^\sigma,\T\cdot\ga_{i_j}^\ell;\K)=\sum\limits_{j=1}^{s}\hbox{dim}\;C_q(\Sk,\T\cdot\ga_{i_j}^\ell;\K).
\end{equation}

Due to  $\H_q(\Lambda_{\al^{\ell \rho}} M;\K)\neq 0$,  $\hbox{dim}\;{\rm H}_q(\M_{\al^{\ell\rho}};\K)\geq p$ for some $p\in\bbN$. So we can choose $\K$-linearly independent elements of ${\rm H}_q(\M_{\al^{\ell\rho}};\K)$, saying $u_1,\ldots, u_p$, which have representations of  singular cycles $Z_1,\ldots, Z_p$ of ${\rm H}_q(\M_{\al^{\ell\rho}};\K)$. Let $\Sigma$ be a compact set which contains the supports of  $Z_i,\;i=1,\ldots, p$. Then there exists $B\in\R$ such that for any $b\geq B$ we have $\Sigma\subseteq (f_\eps)^b$. Clearly, $Z_1,\ldots, Z_p$ are also singular cycles of $\Sigma\subseteq (f_\eps)^b$ which are $\K$-linearly independent.  So we have
\begin{equation}\label{e:nontrivialhomo}
\hbox{dim}\;{\rm H}_q\big( (f_\eps)^b;\K)\big)\geq p.
\end{equation}

Now we choose $A>B$ such that $A>B+\delta+\eps$. It follows from (\ref{e:nontrivialCrgroup})
and (\ref{e:nontrivialhomo}) that

\begin{equation}\label{e:nontrivial}
\sum\limits_{j=1}^{s}\hbox{dim}\;C_q(\Sk,\T\cdot\ga_{i_j}^\ell;\K)\geq p.
\end{equation}
So $\hbox{dim}\;C_q(\Sk,\T\cdot\ga_{i_j}^\ell;\K)\geq 1$ for some $j$. It is well known that the critical group $C_d(\Sk,\T\cdot\ga)$ vanishes if $d\notin [m(\ga),m^*(\ga)]$. This completes the proof of the claim.

\noindent \textbf{Continuing the proof of Theorem~\ref{thm: mainresult2}}. We shall prove
\begin{clm}~\label{clm:vanishing} For every integer $q\geq m+2$ there exists a constant $\ell_0(q)>0$ such that
	for any $\ell\geq \ell_0(q)$,
	$$C_q\big(\Sk,\T\cdot\ga_i^\ell;\K\big)=0\quad\forall i=1,\ldots, r.$$
\end{clm}

\noindent {\bf Proof of Claim~\ref{clm:vanishing}.}
Equip the Hilbert manifold $W^{1,2}(\T,M)$ a new metric as
\begin{equation}\label{e:NewRiemetric}
\langle \xi,\eta\rangle'_x:=\int_{\T}\langle \xi(t), \eta(t)\rangle_{x(t)}dt+\int_{\T}\langle \nat\xi(t), \nat\eta(t)\rangle_{x(t)} dt.
\end{equation}
Clearly, this metric is equivalent to the metric (\ref{e:Riemetric}), and induces a Riemannian structure on $\M=W^{1,2}(\T,M)\times \R^+$ given by the product metric
\begin{equation}\label{e:prodmetr'}
\big\langle(\xi,\al),(\eta,\beta)\big\rangle'_{(x,T)}=\al\beta+\langle \xi,\eta\rangle'_x
\end{equation}
which is equivalent to the original metric (\ref{e:prodmetr}) on $\M$. Under the metric~(\ref{e:prodmetr'}), the action of $\T$ given by~(\ref{e:invariant}) is a smooth isometric action on $\M$ (hence on each component $\M_\al$ of $\M$).

Now we consider the $\T$-invariant function $\Sk$ on the  $\T$-Hilbert manifold $\M_{\al^{\ell\rho}}$ equipped with the metric~(\ref{e:prodmetr'}). Then we have
\begin{equation}\label{e:Taction}
\big\langle d^2\Sk(\psi (t, z))d\psi_tX,Y\big\rangle=\big\langle d\psi_t d^2\Sk(z)X,Y\big\rangle,\quad \forall X\in T_z\M,\;Y\in T_{\psi_t(z)}\M,
\end{equation}
where $\psi_t(\cdot)=\psi(t,\cdot)$ is the $\T$-action, $d^2\Sk(z)$ is the bounded self-adjoint operator associated to the second differential of $\Sk(z)$.

By the $\T$-tubular neighborhood theorem~\cite{Bre}, for every isolated critical orbit $N_i:=\T\cdot \ga_i^\ell$ in $\mathcal{K}(\Sk, \al^{\ell\rho})$  (see~(\ref{e:iteration})), we take a $\T$-tubular neighborhood $B^i_\varepsilon$ of $\T\cdot \ga_i^\ell$ such that $B^i_\varepsilon$ is diffeomorphic to $\nu N_i(\varepsilon)$, where $B^i_\varepsilon=\{z\in\M_{\al^{\ell\rho}}|\hbox{dist}(z,N_i)\leq  \varepsilon\}$, $\nu N_i$ is the normal bundle of $N_i$ and $\nu N_i(\varepsilon)=\{(z,v)\in\nu N_i|z\in N_i,\;\|v\|\leq\varepsilon\}$.
The diffeomorphism between $B^i_\varepsilon$ and $\nu N_i(\varepsilon)$ is $\T$-equivariant. For a sufficiently small $\varepsilon>0$ we have that all $B^i_\varepsilon$ are mutually disjoint in $\M_{\al^{\ell\rho}}$. By~(\ref{e:prodmetr'}), for any $z=(x,T)\in B^i_\varepsilon$ the image $x(\T)$  is lying in some compact neighborhood of the image $\ga_i([0,T_i])$ in $M$. So, for every $k>c_u(L)$, $\Sk$ satisfies the Palais-Smale condition on all  $B^i_\varepsilon$, $i\in\{1,\ldots,r\}$, see~\cite{CIPP2}. Since we always consider the local homological invariants $C_*(\Sk, N_i;\K)$, for the simplicity of notations, we may identify each $B^i_\varepsilon$ with $\nu N_i(\varepsilon)$, and often work on $\nu N_i(\varepsilon)$ if there is no confusion.

Since $0$ is an isolated eigenvalue of $d^2\Sk(\ga_i^\ell)$, for every $i\in\{1,\ldots,r\}$,  by~(\ref{e:Taction}) we have an orthogonal composition of $\nu N_i$
$$\nu N_i=\nu^+N_i\oplus\nu^-N_i\oplus\nu^0N_i,$$
where $\nu^\pm N_i$ corresponds to the positively/negatively definite space of $d^2\Sk$, and $\nu^0N_i$ corresponds to the null space of $d^2\Sk$. All the three bundles are all $\T$-Hilbert vector bundles.

Denote $\widetilde{\mathcal{S}}_k=\Sk|_{\nu^0 N_i(\varepsilon)}$, then every $N_i$ is an isolated critical orbit of $\widetilde{\mathcal{S}}_k$ on $\nu^0 N_i(\varepsilon)$. So all  $C_*(\widetilde{\mathcal{S}}_k,N_i;\K)$ are well-defined.

Since $\Sk\in C^\infty(\nu N_i(\varepsilon),\R)$ is $\T$-invariant and satisfies the Palais-Smale condition, by shifting theorem in~\cite[Corollary~2.5]{Wa} or \cite[Theorem~2.4]{Cha} we have
\begin{equation}\label{e:shiftthm}
C_*\big(\Sk, \T\cdot \ga_i^\ell;\K\big)=C_{*-m(\ga_i^\ell)}\big(\widetilde{\mathcal{S}}_k,\T\cdot \ga_i^\ell;\K\big).
\end{equation}

If $\hat{m}(\ga_i)=0$, by (\ref{e:iterationineq1}), we get
$0\leq m(\ga_i^\ell)\leq m+1-m^0(\ga_i^\ell)$. Thus for $q\geq m+2$, we have
$$q-m(\ga_i^\ell)\geq q-m-1+m^0(\ga_i^\ell)\geq 1+m^0(\ga_i^\ell)$$
Since $\hbox{dim}\;\nu^0N_i=m^0(\ga_i^\ell)$,
$\widetilde{\mathcal{S}}_k$ is defined on a manifold of dimension $m^0(\ga_i^\ell)\leq 2m$. It follows that
$$C_{q-m(\ga_i^\ell)}(\widetilde{\mathcal{S}}_k,\T\cdot \ga_i^\ell;\K)=0.$$

If $\hat{m}(\ga_i)>0$, by (\ref{e:iterationineq1}), we get
$\ell\hat{m}(\ga_i)-m\leq m(\ga_i^\ell)$. So we have
$$q-m(\ga_i^\ell)\leq q+m-\ell \hat{m}(\ga_i)<0$$
provided that $\ell>(q+m)/\hat{m}(\ga_i)$, which implies
$$C_{q-m(\ga_i^\ell)}(\widetilde{\mathcal{S}}_k,\T\cdot \ga_i^\ell;\K)=0.$$
For $q>m+2$, we take
$$\ell_0(q)=1+\max\bigg\{\bigg[\frac{q+m}{\hat{m}(\ga_i)}\bigg]\bigg|\hat{m}(\ga_i)\neq0,\;1\leq i\leq r\bigg\}.$$
Here $[a]$ denotes the largest integer less than or equal to $a$.
Then the desired claim follows from (\ref{e:shiftthm}) immediately.

Finally, if $q> 2m+1$, from Claim~\ref{clm:novanishing}
and Claim~\ref{clm:vanishing} we conclude a contradiction. So the assumption~(H) does not hold. This completes the proof.

\end{proof}

\section*{Acknowledgements}

The author would like to express his great gratitude to the anonymous
referees for his carefully reading and helpful suggestions which improve the elaboration of the present paper. He also thanks his colleague  Hong Huang for helpful discussions.
The author is supported by the National Natural Science Foundation of China No. 11701313.

\section{Appendix~A: an example without closed magnetic geodesics in high energy levels}\label{app:Nonorbits}
Consider the cylinder $\Sigma:=\R\times\R/\Z$ with the metric $g$, which, in local coordinates $(x,y)$, is given by $dx^2+(1+e^x)^2 dy^2$.
Let $L:T\Sigma\to\R$ be a Lagrangian of the form
$$L(q,v)=\frac{1}{2}|v|_g^2+\theta(v),$$
where $|\cdot|_g$ is the norm induced by $g$, and $\theta$ is a one form on $\Sigma$ given by
\begin{equation}\label{e:1form}
\theta=(e^x+1)dy.
\end{equation}

We shall prove
\begin{prop}\label{prop:norbit}
	For every $k>c(L)$ (hence $k>c_u(L)$), there are no closed orbits of energy $k$  in any non-trivial free homotopy class for the magnetic flow of the pair $(g,-d\theta)$ on $\R\times\R/\Z$.
\end{prop}

Recall that the hamiltonian associated to $L$ is given by
$$H:T^*\Sigma\longrightarrow\R,\qquad H(q,p)=\frac{1}{2}|p-\theta_q|^2_{g^*},$$
where $g^*$ is the induced metric by $g$ on the cotangent space $T^*\Sigma$.

We also recall that the critical value $c(L)$ has an equivalent definition (see~\cite{FM}):
\begin{equation}\label{e:equidef}
c(L)=\inf\limits_{u\in C^\infty(\Sigma,\R)}\sup\limits_{q\in\Sigma}H(q,d_qu).
\end{equation}
We show that $c(L)=\frac{1}{2}$.  By choosing a constant function $u$, it follows from (\ref{e:equidef}) that
$$c(L)\leq \frac{1}{2}\sup\limits_{q\in \Sigma}|(e^x+1)dy|^2_{g^*}=\frac{1}{2}.
$$
Next, we pick a closed curve $\ga: t\mapsto (r,-t)\in \R\times\R/\Z$ with $t\in[0,1]$, then
\begin{equation}\label{e:app1}
S_{L+k}(\ga)=\frac{1}{2}(1+e^r)^2-e^r-1+k.
\end{equation}
If $k<\frac{1}{2}$,  from (\ref{e:app1})  we deduce that for $r$ small enough $S_{L+k}(\ga)<0$. So $c(L)\geq\frac{1}{2}$, and thus $c(L)=\frac{1}{2}$.

The Christoffel symbols of the Levi-Civita connection of the metric $g$ are computed as follows:
\begin{gather}
\Gamma^1_{11}=\Gamma^1_{12}=\Gamma^1_{21}=\Gamma^2_{11}=\Gamma^2_{22}=0, \label{e:Ch1}\\
\Gamma^1_{22}=-e^x(e^x+1),\quad\Gamma^2_{12}=\Gamma^2_{21}=\frac{e^x}{e^x+1}. \label{e:Ch2}
\end{gather}

Set $$\beta(x)=e^x+1.$$

\begin{proof}[Proof of Proposition~\ref{prop:norbit}]
	Assume that $\ga(t)=(x(t),y(t))$ is a closed magnetic geodesic of energy $k$ with $k>c(L)=\frac{1}{2}$,  i.e., $\ga$ solves the equations
	\begin{gather}
	\nat\dot{\ga}=Y(\dot{\ga})\label{e:orbit1}
	\quad \hbox{and} \\
	\frac{1}{2}|\dot{\ga}(t)|_g^2=k,\;\forall\;t\in [0,T]\label{e:obit2}.
	\end{gather}
By (\ref{df:Loforce}) and  (\ref{e:1form}), we find that
\begin{equation}\label{e:Y}
Y(u)=-\frac{\beta'}{\beta}Ju,
\end{equation}
where $J$ is the almost complex structure compatible with $g$, and satisfies
$$J\frac{\partial}{\partial x}= \frac{1}{\beta}\frac{\partial}{\partial y}\quad\hbox{and}\quad  J\frac{\partial}{\partial y}=-\beta\frac{\partial}{\partial x}.$$
From (\ref{e:Ch2}), (\ref{e:orbit1}) and (\ref{e:Y}), it follows that
\begin{equation}\label{e:orbit3}
\ddot{x}=e^x(e^x+1)\dot{y}^2+e^x\dot{y}=\frac{e^x}{e^x+1}\big(\beta^2\dot{y}^2+\beta\dot{y}\big).
\end{equation}
Since $\ga(t)$ is a periodic orbit, there exists a maximum $t'$ of $x(t)$, hence $x(t')=0$. Then, by~(\ref{e:obit2}), we see that
\begin{equation}\notag
\big(\beta\dot{y}\big)^2(t')=2k
\end{equation}
which implies
\begin{equation}\label{e:convex}
\big(\beta^2\dot{y}^2+\beta\dot{y}\big)(t')>0
\end{equation}
provided that $k>\frac{1}{2}$. Combining (\ref{e:orbit3}) with (\ref{e:convex}) yields $\ddot{x}(t')>0$  which implies that
$t'$ is not a maximum in contradiction with our assumption.
	
\end{proof}

\section{Appendix~B: the proof of Lemma~\ref{lem:Crgroup}}\label{app:lem9}
Following the line of the proofs in ~\cite[Theorem~2.1]{Wa} and~\cite[Theorem~I.4.2]{Cha}, the key idea to prove Lemma~\ref{lem:Crgroup} is to use the following deformation lemma.

\begin{lem}[Deformation lemma]\label{lem:deform}
Let $\M$ be a $C^2$-Riemannian-Hilbert manifold and let $f\in C^2(\M,\mathbb{R})$ satisfying  the Palais-Smale condition on $f^b_a$ with $a<b$. Suppose that $c\in (a,b)$ is an isolated critical value of $f$ in  $[a,b)$, and the connected components of $K_c:=f^{c}\cap \mathcal{K}(f)$ consist of isolated critical submanifolds $N_i, i=1,\ldots, k.$ If the gradient flow of $f$ near every $N_i$ preserves the fibers of the normal bundle of every $N_i$, then $f^c$ is a strong deformation retract of $f^b\setminus K_b$, i.e., there is a continuous mapping $\tau:[0,1]\times f^b\setminus K_b\to f^b\setminus K_b$ such that
\begin{itemize}
	\item[(1) ]$\tau(0,\cdot)=id$;
	\item[(2) ] $\tau(t,\cdot)|_{f^c}=id|_{f^c}$;
	\item[(3) ]$\tau(1,x)\in f^c,\;\forall x\in f^b\setminus K_b$.
\end{itemize}

\end{lem}

We assume Lemma~\ref{lem:deform} at this point, then we can prove Lemma~\ref{lem:Crgroup} as follows.

\begin{proof}[{Proof of Lemma~\ref{lem:Crgroup}}]
By Lemma~\ref{lem:deform},  $f^c$ is a strong deformation retract of $f^b$, and $f^a$ of $f^c\setminus K_c$. So we have
\[
H_*(f^b,f^a;\K)\cong H_*(f^c,f^c\setminus K_c;\K)\cong\bigoplus_{i=1}^r H_*(U_i\cap f^c,U_i\cap (f^c\setminus N_i);\K)=\bigoplus_{i=1}^rC_q(f, N_i;\K),
\]
where we have used the excision theorem of the singular homology theory in the second isomorphism.
\end{proof}

We shall use the following crucial assumption and lemma to prove the deformation lemma.

Let $N$ be an isolated critical submanifold of $f$ and  $B$ a tubular neighborhood of $N$ (being viewed as a normal  bundle $\nu N\to N$ in $\M$).  Consider the flow $\eta(t,u)$ defined by the equation
\begin{equation}
\begin{cases}
\frac{d\eta(t,u)}{dt}=-df\big((\eta(t,u))\big)\\
\eta(0,u)=u\in B.
\end{cases}
\end{equation}

\noindent {\bf Assumption A.} The flow preserves the fibers of $B$, that is, if there exists $t_0>0$ such that $\eta(u,t)\in B$ for any $t\in [0,t_0]$ and $\eta(0,u)\in B_y$ for some $y\in N$, then $\eta(t,u)\in B_y$ for any $t\in[0,t_0]$.

\noindent {\bf Lemma B. (see~[13, Chapter I.~Lemma 3.3])} Suppose that $K$ is a compact metric space, $F_1,F_2$ are compact subsets of $K$. Then either there is a connected component $K$ which connects $F_1$ with $F_2$, or there are compact subsets $M_1$ and $M_2$ in $K$ such that
$M_1\cap M_2=\emptyset$, $M_1\cup M_2=K$, $F_i\subset M_i$,
$i=1,2$.

\begin{proof}[{Proof of Lemma~\ref{lem:deform}}]
For each $x\in f^b\setminus (f^c\cup K_b)$, we define a flow
\begin{equation}
\begin{cases}
\dot{\eta}(t,x)=-\frac{df(\eta(t,x))}{\|df(\eta(t,x))\|^2}\\
\eta(0,x)=x
\end{cases}
\end{equation}
Then we have
\[f(\eta(t,x))=f(x)-t.\]
Let $[0,T(x))$ be the maximal solvable half interval. Then we have
$T(x)=f(x)-c$ and
\[\lim_{t\to T(x)}f(\eta(t,x))=c.\]
{\bf  Claim 1.}
$\lim_{t\to T(x)-0}\eta(t,x)$ exists, and hence $\eta(t,x)$ can be extended to $[0,T(x)]$ such that
$f(\eta(T(x),x))=c$.\\
To prove {\bf  Claim 1} we set
\[\alpha=\inf \limits_{t\in[0,T(x))}\hbox{dist}(\eta(t,x),K_c)\]
and discuss in two cases. If $\alpha>0$, by (PS) condition, there exists $\beta>0$ such that
\[
\inf \limits_{t\in[0,T(x))} \|df(\eta(t,x))\|\geq \beta.
\]
So we have
\[
\hbox{dist}\big(\eta(t_1,x),\eta(t_2,x)\big)\leq \int^{t_2}_{t_1}\bigg\|\frac{d\eta}{dt}\bigg\|dt\leq \frac{t_2-t_1}{\beta}.
\]
Since $T(x)$ is finite, $\lim_{t\to T(x)-0}\eta(t,x)$ exists.\\

For case
$\alpha=0$ we first show that
\begin{equation}\label{e:dist}
\lim_{t\to T(x)-0} \hbox{dist}\big(\eta(t,x),K_c\big)=0.
\end{equation}
If  not, there exists $t_i\to T(x)-0$ such that
\[
\hbox{dist}\big(\eta(t_i,x),K_c\big)\geq \epsilon>0
\]
On the other hand, in case $\alpha=0$ we have $t_i'\to T(x)-0$ such that
\[
\hbox{dist}\big(\eta(t_i',x),K_c\big)=0.
\]
Therefore, we get two sequences $t_i^*<t^{**}_i$, both converging to $T(x)$ such that
\[
\hbox{dist}\big(\eta(t_i^*,x),K_c\big)=\frac{\epsilon}{2},\quad
\hbox{dist}\big(\eta(t_i^{**},x),K_c\big)=\epsilon
\]
and
\[\eta(t,x)\in (\bar{K}_c)_\epsilon\setminus(K_c)_{\epsilon/2}\quad \forall t\in[t_i^*,t^{**}_i],\]
where $(K_c)_\delta$ denotes the $\delta$-neighborhood of $K_c$.

By (PS) condition, we have
\[
\inf\limits_{t\in[t_i^*,t_i^{**}]}\big\|df(\eta(t,x))\big\|\geq \gamma>0.
\]
Hence, we get
\[
\frac{\epsilon}{2}\leq \hbox{dist}\big(\eta(t_i^{**},x),\eta(t_i^{*},x)\big)
\leq \int^{t_i^{**}}_{t_i^*}\bigg\|\frac{d\eta}{dt}\bigg\|dt\leq
\frac{1}{\gamma}|t^{**}_i-t^*_i|\to 0.
\]
This contradiction implies (\ref{e:dist}).\\

Using (PS) condition again, for any sequence $t_i\to T(x)-0$, there is a convergent subsequence of $\eta(t_i,x)$. We claim that the limit set $\Lambda$ of $\{\eta(t,x):\; t\in[0,T(x))\}$ is a compact connected  subset of $K_c$. The compactness is obvious. We only need to prove the connectedness. If not, there exist open subsets $U$ and $V$ such that
\[
U\cap V=\emptyset,\quad \Lambda=(U\cap \Lambda)\cup(V\cap \Lambda),\quad U\cap \Lambda\neq\emptyset,\quad V\cap\Lambda\neq\emptyset.
\]
Pick $z\in U\cap \Lambda$ and $z'\in V\cap\Lambda$. Then there exist $t_i\to T(x)-0$ and $t'_i\to T(x)-0$ such that $\eta(t_i,x)\to z$ and
$\eta(t_i',x)\to z'$. For sufficiently large $i$, we have
\[\eta(t_i,x)\in U,\quad \eta(t'_i,x)\in V.\]
Due to the connectedness of $[t_i,t_i']$, there exists $t_i^*\in[t_i,t_i']$ (or $[t_i',t_i]$) such that
\[\eta(t^*_i,x)\notin U\cup V.\]
By (PS) condition, $\eta(t^*_i,x)$ has a limit point $z^*\in A$. But $z^*\notin U\cup V$-- a contradition, and thus $\Lambda$ is connected.

Now by the assumption about $K_c$, $\Lambda$ is a part of a certain isolated critical submanifold, say $N$. In what follows we show that
$\Lambda$ is indeed a point. We take a tubular neighborhood $B$ of $N$ so that {\bf Assumption A} holds for $B$. Clearly,
\begin{equation}\label{e:distN}
\lim_{t\to T(x)-0} \hbox{dist}(\eta(t,x),N)=0.
\end{equation}
So there is a $\delta>0$ such that $\eta(t,x)\in B$ for any $t\in[T(x)-\delta, T(x))$. Since the flow $\eta$ preserves each fiber $B_x$ of $B$, there exists $z\in\Lambda$ such that $\eta(t,x)\in B_z$ for all $t\in[T(x)-\delta, T(x))$. Notice that $N\cap B_z=\{z\}$, we deduce from (\ref{e:distN}) that
\[
\lim_{t\to T(x)-0}\eta(t,x)=z.
\]
Summing up the above two cases we complete the proof of {\bf Claim~1}.\\

Now we define the deformation retract map
\begin{equation}\notag
\tau(t,x)=\begin{cases}
\eta(tT(x),x)\quad &(t,x)\in[0,1)\times (f^b\setminus (f^c\cup K_b))\\
\lim\limits_{s\to 1-0}\eta(sT(x),x)\quad &(t,x)\in\{1\}\times (f^b\setminus (f^c\cup K_b))\\
x,\quad &(t,x)\in [0,1]\times f^c.
\end{cases}
\end{equation}
To verify the continuity of $\tau$, we consider the following cases:

\begin{itemize}
	\item[(a)] $(t,x)\in [0,1]\times (f^c)^\circ$;
	\item[(b)] $(t,x)\in [0,1)\times (f^b\setminus(f^c\cup K_b))$;
	\item[(c)] $(t,x)\in\{1\}\times (f^b\setminus (f^c\cup K_b))$;
	\item[(d)]$(t,x)\in [0,1]\times f^{-1}(c)$.
\end{itemize}
Case (a) is trivial. Case (b) is implied by the fundamental theorem of O.D.E. The proofs of Case (c) and Case (d) are similar. We only verify Case (c).

Let $x_0\in f^b\setminus (f^c\cup K_b)$, and let $z=\eta(T(x_0),x_0)\in K_c$. Then $z$ belongs to one of isolated critical submanifolds in $K_c$ which we denote by $N$ for simplicity.
Let $D(\epsilon_0)=\nu N(\epsilon_0)$ be the disc normal bundle of $N$ of radius $\epsilon_0>0$ in $\M$, being viewed as a tubular neighborhood of $N$, which satisfies {\bf Assumption A}.

\noindent {\bf Claim~2.} For any given $\epsilon\in (0,\epsilon_0]$, there is a $\delta>0$ such that
\[
\eta(t,x)\in D(\epsilon),\quad \hbox{if}\; t\in [T(x)-\delta, T(x))\;\hbox{and}\; \hbox{dist}(x,x_0)<\delta,
\]
where $D(\epsilon)=\nu N(\epsilon)$.

Assume that {\bf Claim~2} does not hold. Then there is an $\epsilon_1$ and sequences $t_n\to T(x_0)-0$, $x_n\to x_0$ as $n\to\infty$ such that
\begin{equation}\label{e:B}
\eta(t_n,x_n)\notin D(\epsilon_1).
\end{equation}

Let $F_1=\{N\}$ and $F_2=(\M\setminus D(\epsilon_0/2))\cap K_c$, then both $F_1$ and $F_2$ are compact subsets of $K_c$. If $F_1$ and $F_2$ are both nonempty, then from {\bf Lemma B} and the assumption about $K_c$, there are two compact subsets of $K_c$ denoted by $M_1$ and $M_2$ such that
\[
M_1\cup M_2=K_c,\; M_1\cap M_2=\emptyset,\; F_i\subset M_i,\;i=1,2.
\]
Clearly, $\hbox{dist}(M_1,M_2)>0$. If $F_2$ is empty, we let $M_1=K_c$ and $M_2=\emptyset$. \\
Set $E=M_2\cup (\M\setminus D(\epsilon_0))$, then $\alpha=\hbox{dist}(E,M_1)>0$. Without loss of generality,  we assume that
\[
\epsilon_1\leq \min\{\alpha /4,\epsilon_0/4\}.
\]

Pick $\delta_1>0$ such that $\hbox{dist}(\eta(t,x_0),z)<\epsilon_1/8$ for $t\in [T(x_0)-\delta_1,T(x_0))$. Then we have $\delta_2>0$ such that
$\hbox{dist}(\eta(t,x_0),\eta(t,x))<\epsilon_1/8$ for $t\in[0,T(x_0)-\delta_1]$ and $\hbox{dist}(x,x_0)<\delta_2$. So
we have
$\hbox{dist}(\eta(T(x_0)-\delta_1,x),z)<\epsilon_1/4$ for $\hbox{dist}(x,x_0)<\delta_2$. Repeating this process, one can find a subsequence of $x_n$ (still denoted by $x_n$) and another sequence $t_n'\to T(x_0)-0$, if $n\to\infty$, such that
\begin{equation}\label{e:t'}
\eta(t_n',x_n)\in D(\epsilon_1/4).
\end{equation}
This, together with (\ref{e:B}), implies that one can find two sequences $s_n',s_n''$ with $s_n', \to T(x_0)-0$ and $s_n''\to T(x_0)-0$ so that
\[
\hbox{dist}(\eta(s_n',x_n),F_1)=\epsilon_1\quad
\hbox{dist}(\eta(s_n'',x_n),E)=\epsilon_1
\]
and
\[
\eta(t,x_n)\notin (F_1)_{\epsilon_1}\cup (E)_{\epsilon_1}\; \forall t\in[s_n',s_n'']\quad n=1,2,\cdots
\]
By (PS) condition we have that
\[
\gamma=\inf_{x\in A}\|df(x)\|>0,\quad A=f^{-1}([a,b))\setminus
\big( (F_1)_{\epsilon_1}\cup (E)_{\epsilon_1}\big)
\]
It follows that
\[
\epsilon_1\leq \hbox{dist}\big(\eta(s_n',x_n),\eta(s_n'',x_n)\big)\leq \frac{1}{\gamma}|s_n'-s_n''|\to 0
\]
for $n\to\infty$. This is a contradiction, and hence {\bf Claim~2} is true.

Now we can prove $\tau$ is continuous at $(1,x_0)$. If not, there is
$\epsilon_2>0$ and $t_n,x_n$ satisfying $t_n\to T(x_0)-0$, $x_n\to x_0$ as $n\to\infty$ such that
\begin{equation}\label{e:z}
\hbox{dist}\big(\eta(t_n,x_n),z\big)\geq 2\epsilon_2.
\end{equation}
Here we ask $\epsilon_2\leq \epsilon_0/2$. By {\bf Claim~2} there is a $\delta>0$ such that
\[\eta(t,x)\in D(\epsilon_2)\;\hbox{for}\; t\in[T(x_0)-\delta,T(x_0)),\;\hbox{dist}(x,x_0)<\delta.\]
Hence, for $n$ large enough we have
\begin{equation}\label{e:seq}
\eta(t_n,x_n)\in D(\epsilon_2).
\end{equation}
From this and (\ref{e:z}) we get
\[
\hbox{dist}\big(\pi(\eta(t_n,x_n)),z\big)\geq \epsilon_2
\]
where $\pi:\nu N\to N$ is the bundle projection. Let $z_n=\pi(\eta(t_n,x_n))$. It follows from (\ref{e:seq}) and {\bf Assumption A} that if $n$ is large enough then for all $t\in[t_n,T(x_n))$, $\eta(t,x_n)$ belong to the fiber $B_{z_n}$, and hence
\[
\hbox{dist}\big(\eta(t,x_n),z\big)\geq \epsilon_2\quad \forall t\in[t_n,T(x_n)).
\]
However, similar to the proof of (\ref{e:t'}), for fixed $t'$ (approaching to $T(x_0)$) one can find $\delta'>0$ such that
\[
\hbox{dist}(\eta(t',x_n),z)< \epsilon_2\;\hbox{for}\;
\hbox{dist}(x_n,x_0)<\delta'.
\]
Therefore, for each $n$ there is a $t_n'$ with $t_n'<t_n$ and $t_n'\to T(x_0)$ as $n\to \infty$ such that $\eta(t_n',x_n)\notin D(\epsilon_0)$. Then similar to the case that (\ref{e:B}) holds, we arrive at a contradiction. This completes the proof of the continuity of $\tau$. Clearly, the properties (1)-(3) of the deformation retract $\tau$ are satisfied.

\end{proof}

\end{document}